\documentclass[a4paper,reqno]{amsart}

\usepackage[utf8]{inputenc}
\usepackage[english]{babel}

\usepackage{amsmath,amssymb,amsfonts,amsthm}
\usepackage{mathrsfs, esint, dsfont}
\usepackage{moreverb,rotating,graphics,float, subcaption}
\usepackage[colorlinks, linkcolor={blue}, citecolor={red}]{hyperref}
\usepackage{framed,fancybox}
\usepackage{enumerate}
\usepackage{csquotes}
\usepackage{float}
\usepackage{slashed}

\newtheorem{theorem}{Theorem}

\newtheorem{remark}[theorem]{Remark}
\newtheorem{lemma}[theorem]{Lemma}
\newtheorem{corollary}[theorem]{Corollary}
\newtheorem{definition}[theorem]{Definition}

\newtheorem{example}[theorem]{Example}

\newcommand\1{{\mathds 1}}

\def\C{{\mathbb C}}

\def\bbI{{\mathbb I}}

\def\N{{\mathbb N}}

\def\R{{\mathbb R}}

\def\SS{{\mathbb S}}
\def\TT{{\mathbb T}}

\def\Z{{\mathbb Z}}
\def\ZZ{{\mathbb Z}}

\def\bA{{\mathbf A}}

\def\bk{{\mathbf k}}

\def\bx{{\mathbf x}}

\def\by{{\mathbf y}}
\def\bz{{\mathbf z}}

\def\bnull{{\mathbf 0}}

\def\rd{{\mathrm{d}}}
\def\re{{\mathrm{e}}}
\def\ri{{\mathrm{i}}}

\def\cB{{\mathcal B}}

\def\cD{{\mathcal D}}

\def\cH{{\mathcal H}}
\def\cI{{\mathcal I}}

\def\cN{{\mathcal N}}

\def\cS{{\mathcal S}}

\def\cU{{\mathcal U}}

\newcommand{\Tr}{{\rm Tr}}

\newcommand{\bra}{\langle}
\newcommand{\ket}{\rangle}

\renewcommand{\epsilon}{\varepsilon}

\def\Ker{{\rm Ker}}

\def\U{{\rm U}}
\def\Mas{{\rm Mas}}
\def\Sf{{\rm Sf}}
\def\Winding{{\rm Winding}}

\author{David Gontier}
\address{CEREMADE, University of Paris-Dauphine, PSL University, 75016 Paris, France}
\email{gontier@ceremade.dauphine.fr}

\title[Edge states for elliptic operators in a channel]{Edge states for second order elliptic operators in a channel}

\date{\today}

\begin{document}

\begin{abstract}
We present a general framework to study edge states for second order elliptic operators in a half channel. We associate an integer valued index to some bulk materials, and we prove that for any junction between two such materials, localised states must appear at the boundary whenever the indices differ.

\bigskip

\noindent \sl \copyright~2022 by the author. This paper may be reproduced, in its entirety, for non-commercial purposes.
\end{abstract}

\maketitle


\section{Introduction and statement of the main results}

Bulk-edge correspondence states that one can associate an integer valued index $\cI \in \Z$ to some bulk materials (represented here by Schrödinger (PDE) or Hill's (ODE) operators). When the material is cut, edge states appear at the boundary whenever $\cI \neq 0$. In addition, it is believed that any junction between a left and a right material having indices $\cI_L$ and $\cI_R$ must also have edge states near the junction whenever $\cI_L \neq \cI_R$. We prove this fact in this paper.

Since the original works of Hatsugai~\cite{hatsugai1993chern, hatsugai1993edge}, most studies on bulk-edge correspondence focused on tight-binding models ({\em e.g.}~\cite{graf2013bulk, avila2013topological}), set on half-spaces. In these tight-binding models, boundary conditions at the cut are quite simple to describe, and it turns out that the index is independent of these boundary conditions. In the context of continuous models, it is unclear that one can define an index which is indeed independent of the chosen boundary conditions. In~\cite{Gontier2020edge}, we proved that it was the case in a simple one-dimensional model for dislocations. We extend this work here, and give a general framework to define the edge index for different self-adjoint extensions of Schrödinger operators.

We consider two types of continuous models. In the first part, we study families of Hill's operator (ODE) set on $\C^n$, of the form
\[
    h_t := - \partial_{xx}^2 + V_t(x), \quad \text{acting on} \quad L^2(\R, \C^n),
\]
where $t \mapsto V_t$ is a continuous periodic family of bounded potentials, with values in the set of $n \times n$ hermitian matrices. When $t$ is seen as the time variable, this equation models a Thouless pump~\cite{thouless1983quantization, braunlich2010equivalence}. In the case where $V_t(x) = V(x - t)$, the variable $t$ is interpreted as a dislocation parameter~\cite{drouot2018bulk, Gontier2020edge}. On the second part of the article, we study its PDE version, that is families of Schrödinger's operators of the form
\[
    H_t := - \Delta + V_t(x,y), \quad \text{acting on} \quad L^2 \left( \R \times (0,1)^{d-1}, \C   \right).
\]
Here, $\R \times (0, 1)^{d-1}$ is a tube in $\R^d$, and we impose periodic boundary conditions in the last $(d-1)$-directions. Our setting also allows to treat two-dimensional PDE operators of the form
\[
    \widetilde{H_t} := - \partial_{xx}^2 + (- \ri \partial_y + 2 \pi t)^2 + V(x,y), 
    \quad \text{acting on} \quad L^2 \left( \R \times (0,1), \C   \right),
\]
where $k := 2 \pi t$ is interpreted as the Bloch quasi-momentum in the $y$-direction. Such families of operators appear in the study of two-dimensional materials, once a Bloch transform has been performed in the $y$-direction.

In these models, we interpret the bulk-edge index as the intersection of Lagrangian planes on a {\em boundary space} $\cH_b$. Roughly speaking, this space contains the values $(\psi(0), \psi'(0))$ of the admissible wave-functions $\psi$. In the context of Hill's operators, we take $\cH_b = \C^n \times \C^n$, while for Schrödinger operators, $\cH_b = H^{3/2}(\Gamma)\times H^{1/2}(\Gamma)$, where $\Gamma := \{ 0 \} \times (0, 1)^{d-1}$ is the cut.

The link between edge states and Lagrangian planes was already mentioned {\em e.g.} in~\cite{avila2013topological} for discrete models (tight-binding approximation). Based on the recent developments on Lagrangian planes and second order elliptic operator by Howard, Latushkin and Sukhtayev in a series on papers~\cite{howard2016maslov, howard2017maslov, howard2018maslov} (see~\cite{kostrykin1999kirchhoff} for older results in an ODE setting), we extend the picture to the continuous case. This framework allows in particular to treat the PDE setting following~\cite{latushkin2018maslov}, based on the seminal work of Booß-Bavnbek and Furutani on infinite dimensional Lagrangian planes~\cite{booss1998maslov, booss2013maslov, furutani2004fredholm}.

\medskip

Let us state our main results for Hill's operators. They extend the previous works~\cite{drouot2018bulk, Gontier2020edge} and shed a new light on the results by Bräunlich, Graf and Ortelli in~\cite{braunlich2010equivalence}. Some of our results can already be found in the last article. However, the proofs in~\cite{braunlich2010equivalence} use the notion of {\em frames} of solutions. In the present article, we provide proofs which do not rely on this notion, so that we can generalise them to the Schrödinger case, where one cannot construct such frames of solutions.

 Let $n \in \N \setminus \{0\}$ be fixed, and let
\begin{equation} \label{eq:form_V}
V_{t}(x) := V(t,x) : \TT^1 \times \R \to \cS_n,
\end{equation}
be a periodic family of matrix-valued bounded potentials (which are not necessarily periodic in $x$). Here, $\TT^1 \approx [0, 1]$ is the one-dimensional torus, and $\cS_n$ denotes the set of $n \times n$ hermitian matrices. We assume that $t \mapsto V_t$ is continuous from $\TT^1$ to $L^\infty(\R, \cS_n)$. We consider the family of (bulk) Hill's operators
\[
    h_t := - \partial_{xx}^2 + V_t \quad \text{acting on} \quad L^2(\R, \C^n).
\]
For $E \in \R$, we say that $E$ is in the gap of the family $(h_t)$ if $E \notin \sigma(h_t)$ for all $t \in \TT^1$. We also consider the family of (edge) Hill's operators
\[
h^\sharp_{D, t} := - \partial_{xx}^2 + V_t \quad \text{acting on} \quad L^2(\R^+, \C^n),
\]
with Dirichlet boundary conditions at $x = 0$. While $E$ is not in the spectrum of the bulk operator $\sigma(h_t)$, it may belong to the spectrum of the edge operator $h^\sharp_{D, t}$. In this case, the corresponding eigenstate is called an {\em edge mode}. 

As $t$ runs through $\TT^1 \approx [0, 1]$, a spectral flow may appear for the family $h^\sharp_{D, t}$. We denote by $\Sf(h^\sharp_{D, t}, E, \TT^1)$ the net number of eigenvalues of $h^\sharp_{D, t}$ going {\bf downwards} in the gap where $E$ lies. We define the index of $(h_t)_{t \in \TT^1}$ as this spectral flow:
\[
    \cI(h_t, E) := \Sf \left( h^\sharp_{D, t}, E, \TT^1  \right).
\]
Our main theorem is the following (see Theorem~\ref{th:bec_hill} for the proof in the Hill's case, and Theorem~\ref{th:bec_schrodinger} for the one in the Schrödinger case).

\begin{theorem}[Junctions between two channels]
    Let $t \mapsto V_{R,t}$ and $t \mapsto V_{L,t}$ be two continuous periodic families of bounded potentials on $\R$. Let $E \in \R$ be in the gap of both corresponding (bulk) Hill's operators $(h_{L, t})$ and $(h_{R, t})$. Let $\chi : \R \to [0, 1]$ be any switch function, satisfying $\chi(x) = 1$ for $x < -X$ and $\chi(x)  = 0$ for $x > X$ for some $X > 0$, and let
    \[
        h^\chi_t := - \partial_{xx}^2 + V_{L, t}(x) \chi(x) + V_{R, t}(x) (1 - \chi(x)).
    \]
    Then
    \[
         \boxed{ \Sf \left( h^\chi_{t}, E, \TT^1  \right) = \cI(h_{R, t}, E) - \cI(h_{L,t}, E).}
    \]
\end{theorem}

The operator $h^\chi_t$ is a domain wall operator. On the far left, we see the potential $V_{L, t}$, while on the far right, we see $V_{R, t}$, so this operator models a junction between a left potential and a right one. This theorem states that edge states must appear at the junction if the left and right indices differ.



\subsection*{Plan of the paper}

In Section~\ref{sec:first_fact_Hill}, we recall some basic facts on symplectic spaces and self-adjoint extensions of operators. We then prove our results concerning Hill's operators in Section~\ref{sec:family_Hill}, and explain how to adapt the proofs for Schrödinger operators in Section~\ref{sec:Schrodinger}.

\subsection*{Notation of the paper}

We write $\N := \{ 1, 2, 3, \cdots \}$ and $\N_0 := \{ 0, 1, 2, 3, \cdots \}$. For $z_0 \in \C$ and $r > 0$, we set $B(z, r) := \left\{ z \in \C, \ | z - z_0 | < r\right\}$ the open ball in the complex plane.\\
For $\Omega \subset \R^d$ an open set, we denote by $L^p(\Omega, \C)$ the usual Lebesgue spaces, and by $H^s(\Omega, \C)$ the Sobolev ones. The set $H^s_0(\Omega, \C)$ is the completion of $C^\infty_0(\Omega, \C)$ for the $H^s$ norm.

Let $\cH_1$ and $\cH_2$ be two Hilbert spaces. For a bounded operator $A : \cH_1 \to \cH_2$, its dual $A^*$ is the map from $\cH_2 \to \cH_1$ so that
\[
\forall x_1 \in \cH_1, \ \forall x_2 \in \cH_2, \quad \bra x_2, A x_1 \ket_{\cH_2} = \bra A^* x_2, x_1 \ket_{\cH_1}.
\]
The operator $A$ is {\em unitary} if $A^* A = \bbI_{\cH_1}$ and $A A^* = \bbI_{\cH_2}$. 

For $E$ a Banach space, we say that a map $t \mapsto v(t) \in E$ is continuously differentiable if $v'(t)$ is well-defined in $E$ for all $t$ (that is $\| v'(t) \|_E < \infty$), and if $t \mapsto v'(t)$ is continuous.

\subsection*{Acknowledgements}

This work was funded by the CNRS international cooperation tool PICS. We thank the anonymous referees for their comments and suggestions. We also thank the anonymous editor for pointing out an error in the initial version, and for the very interesting detailed comments and references on boundary triples.

\section{First facts and notations}
\label{sec:first_fact_Hill}

\subsection{Lagrangian planes in complex Hilbert spaces}
\label{ssec:LagrangianPlanes}

Let us first recall some facts on symplectic Banach/Hilbert spaces. In the original work of Maslov~\cite{maslov1972theorie}, popularised by Arnol'd~\cite{arnol1967characteristic}, the authors consider {\em real} Banach spaces $E$. Following the recent developments, we present the theory for complex Banach spaces. 

\subsubsection{Basics in symplectic spaces}

Let $E$ be a complex Banach space. A symplectic form on $E$ is a non degenerate continuous sesquilinear form $\omega : E \times E \to \C$ such that
\[
    \forall x, y \in E, \quad \omega(x, y) = -  \overline{ \omega(y, x)} .
\]
For $\ell$ a linear subspace of $E$, we denote by 
\[
\ell^\circ := \left\{ x \in E, \quad \forall y \in \ell, \quad \omega(x,y) = 0\right\}.
\]
The space $\ell^\circ$ is always closed. Such a subspace is called {\bf isotropic} if $\ell \subset \ell^\circ$, {\bf co-isotropic} if $\ell^\circ \subset \ell$, and {\bf Lagrangian} if $\ell = \ell^\circ$. We also say that $\ell$ is a {\bf Lagrangian plane} in the latter case. The set of all Lagrangian planes of $E$,  sometime called the Lagrangian-Grassmanian, is denoted by $\Lambda(E)$.

\begin{example}[In $\R^{2n}$] \label{ex:symplectic_R2n}
    In the {\em real} Hilbert space $E = \R^n \times \R^n$, the canonical symplectic form is given by (we write $\bx = (x, x')$, $\by = (y, y')$, {\em etc.} the elements in $\R^n \times \R^n$)
    \[  
    \forall \bx, \by \in \R^n \times \R^n, \quad 
    \omega (\bx, \by) := \bra x,  y' \ket_{\R^n} - \bra x',  y \ket_{\R^n},
    \]
    When $n = 1$, the Lagrangian planes are all the one-dimensional linear subspaces of $\R^2$. Conversely, if $(\R^N, \omega)$ is a symplectic space, then $N = 2n$ is even, and all Lagrangian planes are of dimension $n$.
\end{example}

\begin{example}[In $\C^{2n}$] \label{ex:symplectic_C2n} 
    Similarly, in the {\em complex} Hilbert space $\C^{2n}$, the canonical symplectic form is given by (we write again $\bz = (z, z')$ the elements in $\C^n \times \C^n$) 
    \[  
    \forall \bz_1, \bz_2 \in \C^n \times \C^n, \quad
    \omega (\bz_1, \bz_2) := \bra z_1,  z_2' \ket_{\C^n} - \bra z_1',  z_2 \ket_{\C^n}.
    \]
    When $n = 1$ for instance, the Lagrangian planes are the one-dimensional linear spaces $L = {\rm Vect}_\C(\bz)$ with $\bz = (z, z')$ satisfying the extra condition $\overline{z} z' \in \R$. Up to a phase, we may always assume $z \in \R$, in which case $z' \in \R$ as well. So the Lagrangian planes are the one-dimensional subspaces of $\C^2$ of the form ${\rm Vect}_\C(\bz)$ with $\bz \in \R^{2}$.
\end{example}

\begin{example}[In $\C^N$] \label{ex:counterExample_symplectic_CN}
    Another example is given by the symplectic form
    \[
    \forall \bz_1, \bz_2 \in \C^N, \quad \widetilde{\omega}(\bz_1, \bz_2) = \ri \bra \bz_1, \bz_2 \ket_{\C^N}.
    \]
    With this symplectic form, a vector $\bz \in \C^N$ is never isotropic, since $\widetilde{\omega} (\bz, \bz) = \ri \| \bz \|^2 \neq 0$ for $z \neq 0$. In particular, $(\C^N, \widetilde{\omega})$ does not have Lagrangian subspaces.
\end{example}

We record the following result.
\begin{lemma} \label{lem:useful_result_lagrangians}
    If $\ell_1 \subset \ell_1^\circ$ and $\ell_2 \subset \ell_2^\circ$ are two isotropic subspaces with $\ell_1 + \ell_2 = E$, then $\ell_1$ and $\ell_2$ are Lagrangians, and $\ell_1 \oplus \ell_2 = E$.
\end{lemma}

\begin{proof}
    Since $\ell_1 + \ell_2 = E$, we have $\{ 0 \} = \ell_1^\circ \cap \ell_2^\circ$. In particular, $\ell_1 \cap \ell_2 \subset \ell_1^\circ \cap \ell_2^\circ = \{ 0\}$ as well, so $\ell_1 \oplus \ell_2 = E$. Let $x \in \ell_1^\circ \subset E$, and write $x = x_1 + x_2$ with $x_1 \in \ell_1$ and $x_2 \in \ell_2$. Since $\ell_1 \subset \ell_1^\circ$, we have $x_2 = x - x_1 \in \ell_1^\circ$ as well, so $x_2 \in \ell_1^\circ \cap \ell_2 \subset \ell_1^\circ \cap \ell_2^\circ = \{ 0 \}$. This proves that $x = x_1 \in \ell_1$, hence $\ell_1^\circ = \ell_1$. The proof for $\ell_2$ is similar.
\end{proof}


\subsubsection{Lagrangian planes of Hilbert spaces and unitaries.}

In the case where $E = \cH_b$ is a Hilbert space, with inner product $\bra \cdot, \cdot \ket_{\cH_b}$, for all $x \in \cH_b$, the map
\[
T_x : y \mapsto \omega(x, y)
\]
is linear and bounded. So, by Riesz' representation theorem, there exists $v \in \cH_b$ so that $T_x(y) = \bra v, y \ket_{\cH_b}$. We denote by $J^*x := v$ this element. This defines an operator $J^*: \cH_b \to \cH_b$, satisfying
\[
     \forall x,y \in \cH_b, \quad \omega(x, y) = \bra J^* x, y \ket_{\cH_b} = \bra x, J y \ket_{\cH_b}.
\]
In particular, since $\omega$  is bounded, we have
\[
    \| Jy \|_{\cH_b}^2 = \bra Jy, Jy \ket_{\cH_b} = \omega(Jy, y) \le C_\omega \| Jy \| \cdot \| y \|,
\]
so $\| J y \| \le C_\omega \| y \|$, and $J$ is a bounded operator. In addition, from the relation $\omega(x,y) = - \overline{\omega(y, x)}$, we get that 
\[
    \bra x, J y \ket_{\cH_b} = - \bra J x, y \ket_{\cH_b},
\] 
that is $J = - J^*$. Finally, since $\omega$ is not degenerate, we have ${\rm Ker}(J)= \{ 0 \}$.
\begin{example}
    On $\C^{2n}$ with the canonical symplectic form $\omega(x, y) = \bra x_1, y_2 \ket_{\C_n} - \bra x_2, y_1 \ket_{\C^n}$, we have
    \[
    J = \begin{pmatrix}
    0_n & \bbI_n \\ - \bbI_n & 0_n
    \end{pmatrix}.
    \]
\end{example}

Later in the article, we will make the following Assumption A:
\begin{equation*} 
    \textbf{Assumption A:} \quad J^2 = - \bbI_{\cH_b}.
\end{equation*}
In this case, $J$ is bounded skew self-adjoint with $J^2 = - \bbI$, and we have
\begin{equation} \label{eq:decomposition_cHb}
\cH_b =  \Ker \left( J - \ri \right) \oplus  \Ker \left( J + \ri \right) .
\end{equation}
The hermitian form $- \ri \omega$ is positive definite on $\Ker \left( J - \ri \right)$ and negative definite on $\Ker \left( J + \ri \right)$. In addition, for all $x \in \Ker(J - \ri)$ and all $y \in \Ker(J + \ri)$, we have
\begin{equation} \label{eq:cross_terms_omega}
   \omega(x,x) = \ri \| x \|_{\cH_b}^2, \quad
    \omega(y,y)  = -\ri \| y \|_{\cH_b}^2, \quad \text{and} \quad
    \omega(x, y) = 0.
\end{equation}

The following result goes back to Leray in its seminar~\cite{leray1978analyse} (see also~\cite{booss1998maslov} and~\cite[Lemma 2 and 3]{booss2013maslov}). We skip its proof for the sake of brevity.
\begin{lemma} \label{lem:Lagrangian_unitaries_U}
    If Assumption A holds, then there is a one-to-one correspondence between the Lagrangian planes $\ell$ of $\cH_b$ and the unitaries $U$ from $\Ker \left( J - \ri \right)$ to $\Ker \left( J + \ri \right)$, with
    \[
    \ell =  \left\{  x + Ux, \quad x \in  \Ker \left( J - \ri \right)   \right\}.
    \]
\end{lemma}

\begin{corollary}
    If $\dim \Ker \left( J - \ri \right) \neq \dim \Ker \left(J + \ri \right)$, then there are no Lagrangian planes. This happens for instance for the symplectic space $(\C^n, \widetilde{\omega})$, with $\widetilde{\omega}(z, z') =  \ri \bra z, z' \ket_{\C^n}$ (see Example~\ref{ex:counterExample_symplectic_CN}), for which we have $Jz = \ri z$, so $\Ker(J - \ri) = \C^n$ while $\Ker(J + \ri ) = \{ 0 \}$.
\end{corollary}

The next Lemma shows that the crossing of two Lagrangian planes can be read from their respective unitaries (see {\em e.g.} in~\cite[Lemma 2]{booss2013maslov}).

\begin{lemma} \label{lem:KerU1U2-1}
    Let $\ell_1$ and $\ell_2$ be two Lagrangian planes of $\Lambda(\cH_b)$, with corresponding unitaries $U_1$ and $U_2$ from $\Ker(J - \ri)$ to $\Ker(J + \ri)$. Then there is a natural isomorphism
    \[
        \Ker \left( U_2^* U_1 - \bbI_{\Ker(J - \ri )}  \right) \approx \ell_1 \cap \ell_2.
    \]
\end{lemma}

\begin{proof}
    If $x^- \in \Ker\left( J - \ri  \right)$ is such that $U_2^* U_1 x^- = x^-$, then we have $U_1 x^- = U_2 x^-$ in $\Ker(J + \ri)$, so $x := x^- + U_1 x^- = x^- + U_2 x^-$ is in $\ell_1 \cap \ell_2$. Conversely, if $x \in \ell_1 \cap \ell_2$, then, writing $x = x^- + x^+$, we have $U_1 x^- = U_2 x^-$, so $U_2^* U_1 x^- = x^-$.
\end{proof}


\subsubsection{Another unitary}
\label{ssec:anotherUnitary}

In Section~\ref{ssec:Maslov} below, we will consider periodic paths of Lagrangians $\ell_1(t)$ and $\ell_2(t)$, and define the Maslov index of the pair $(\ell_1, \ell_2)$. When $\cH_b$ is finite dimensional, we will prove that it equals the winding number of the determinant of $U_2^*(t) U_1(t)$. Unfortunately, since $U_1$ and $U_2$ are not endomorphism, we cannot split $\det(U_2^* U_1)$ into $\det(U_1) / \det(U_2)$. In this section, we present another one-to-one correspondence between Lagrangian planes and other unitaries (which will be endomorphisms). The results of this section are new to the best of our knowledge.

\medskip

We now make the stronger assumption that $\cH_b$ is of the form $\cH_b = \cH_1 \times \cH_2$, where $\cH_1$ and $\cH_2$ are two complex Hilbert spaces, and that, relative to this decomposition, $J$ is of the form
\begin{equation*} 
    \textbf{Assumption B:} \quad J = \begin{pmatrix}
        0 & V^* \\
        - V & 0
    \end{pmatrix}, \quad \text{for some (fixed) unitary $V : \cH_1 \to \cH_2$}.
\end{equation*}
It implies $J^2 = - \bbI_{\cH_b}$, so {\bf B} is stronger assumption than {\bf A}. Furthermore, we can identify
\[
    \Ker( J - \ri ) = \begin{pmatrix}
        1 \\ \ri V
    \end{pmatrix} \cH_1, \quad \text{and} \quad
    \Ker(J + \ri) = \begin{pmatrix}
    1 \\ - \ri V
    \end{pmatrix} \cH_1.
\]
Defining the maps  $Q_\pm : \cH_1 \to \Ker(J \pm \ri)$ by
\[
    \forall x \in \cH_1, \quad Q_\pm (x) := \frac{1}{\sqrt{2}} \begin{pmatrix}
x \\ \mp \ri V x
\end{pmatrix},
\quad \text{with dual} \quad
Q_\pm^* \begin{pmatrix}
y_1 \\ y_2
\end{pmatrix} = \frac{1}{\sqrt{2}}(y_1 \pm \ri V^* y_2) ,
\]
we can check that $Q_\pm Q_\pm^* = \bbI_{\Ker(K \pm \ri)}$ and $Q_\pm^* Q_\pm = \bbI_{\cH_1}$, so $Q_\pm$ are unitaries. In particular, if $U$ is a unitary from $\Ker(J - \ri)$ to $\Ker(J + \ri)$, then 
\[
\cU := Q_+^* U Q_-
\]
is a unitary from $\cH_1$ to itself, hence an endomorphism. In what follows, we use straight letters $U$ for unitaries from $\Ker(J - \ri) \to \Ker(J + \ri)$, and curly letters $\cU$ for unitaries of $\cH_1$. We therefore proved the following.

\begin{lemma} \label{lem:Lagrangian_unitaries_cU}
    If Assumption B (hence A) holds, then there is a one-to-one correspondence between the Lagrangian planes $\ell$ of $( \cH_1 \times \cH_2, \omega)$ and the unitaries $\cU$ of $\cH_1$, with
    \[
        \ell = \left\{ \begin{pmatrix}
         1 \\ \ri V
        \end{pmatrix} x + \begin{pmatrix}
        1 \\ -\ri V
        \end{pmatrix} \cU x, \quad x \in \cH_1 \right\}.
    \]
    In addition, if $\ell_1$ and $\ell_2$ are two Lagrangian planes with corresponding unitaries $\cU_1$ and $\cU_2$, then there is a natural isomorphism
    \[
        \Ker \left( \cU_2^* \cU_1 - \bbI_{\cH_1}  \right) \approx \ell_1 \cap \ell_2.
    \]
\end{lemma}


\subsection{Self-adjoint extensions of Hill's operators}

We now review some basic facts on self-adjoint operators (see {\em e.g.}~\cite[Chapter X.1]{reed1975fourier} for  a complete introduction). We first recall some general definitions, and then focus on second order elliptic operators. We show the connection with symplectic spaces using the second Green's identity.

\subsubsection{Self-adjoint operators}
Let $\cH$ be a separable Hilbert space, and let $A$ with dense domain $\cD_A$ be any operator on $\cH$. In the sequel, we sometime write $(A, \cD_A)$. The {\bf adjoint} of $(A, \cD_A)$ is denoted by $(A^*, \cD_{A^*})$. 

For $A$ a symmetric, hence closable, operator on $\cH$, we denote by $(A_{\rm min}, \cD_{\rm min})$ its closure. The adjoint of $(A_{\rm min}, \cD_{\rm min}) $ is denoted by $(A_{\rm max}, \cD_{\rm max})$. Since $A$ is symmetric, we have  $A_{\rm min} \subset A_{\rm max}$ ($A_{\rm max}$ is an extension of $A_{\rm min}$). The operator $A_{\rm min}$ is self-adjoint iff $\cD_{\rm min} = \cD_{\rm max}$. Otherwise, any self-adjoint extension of $A$ must be of the form $(\widetilde{A}, \widetilde{\cD})$ with
\[
    A_{\rm min} \subset \widetilde{A} \subset A_{\rm max}, \quad \text{in the sense} \quad \cD_{\rm min} \subset \widetilde{\cD} \subset \cD_{\rm max}.
\]
In particular, once $\cD_{\rm min}$ and $\cD_{\rm max}$ have been identified, the self-adjoint extensions are simply given by domains $\widetilde{\cD}$ with $\cD_{\rm min} \subset \widetilde{\cD} \subset \cD_{\rm max}$, and the operator $\widetilde{A}$ acts on this domain via
\[
    \forall x \in \widetilde{\cD}, \quad \widetilde{A} x := A_{\rm max} x.
\]
We sometime write $(A_{\rm max}, \widetilde{\cD})$ instead of $(\widetilde{A}, \widetilde{\cD})$ to insist that only the domain matters. 

There are several ways to identify the self-adjoint extensions of $A$. The original proof by von Neumann~\cite{neumann1930allgemeine} uses the Cayley transform. As noticed in~\cite[Chapter X.1]{reed1975fourier} following~\cite{dunford1965linear}, the connection with boundary values is not so clear in this approach. Another approach can be found {\em e.g.} in~\cite[Section 3.1]{booss1998maslov}, where the authors give a correspondence between the self-adjoint extensions of $A$ and the Lagrangian planes of the abstract space $\cD_{\rm max} / \cD_{\rm min}$, with the symplectic form
\[
    \forall [x], [y] \in \cD_{\rm max} / \cD_{\rm min}, \quad \omega \left([x], [y]\right) := \bra x, A_{\rm max} y \ket_{\cH} - \bra A_{\rm max} x, y \ket_\cH.
\] 
Again, the connection with boundary conditions is not so clear in this setting.

Here, we follow~\cite{latushkin2018maslov} (see also~\cite{cox2015morse}), which is specific to second order elliptic operators. It uses the second Green's identity.

\subsubsection{Self-adjoint extensions of Hill's operators on the semi line}
\label{ssec:sa_Hill}
We first present the theory in the case where $A = h$ is a second order ODE (Hill's operator). We postpone the analysis for general second order elliptic operator to Section~\ref{sec:Schrodinger} below.

Let $n \in \N$ and let $V: \R \to \cS_n$ be a bounded potential with values in $\cS_n$, the set of $n \times n$ hermitian matrices. We consider the Hill's operator
\[
    h := - \partial_{xx}^2 + V(x) \quad \text{acting on} \quad \cH := L^2(\R, \C^n).
\]
The bulk operator $h$ with core domain $C^\infty_0(\R, \C^n)$ is symmetric. Since the potential $V$ is bounded, the operator $h$ is essentially self-adjoint, with domain (see~\cite[Chapter 4]{kato2013perturbation})
\[
    \cD := \cD_{\rm min} = \cD_{\rm max} = H^2(\R, \C^n).
\]

When restricting this operator to the half line, we obtain the edge operator 
\[
    h^\sharp := - \partial_{xx}^2 + V(x) \quad \text{acting on} \quad  \cH^\sharp := L^2(\R^+, \C^n).
\]
On the core $C^\infty_0(\R^+, \C^n)$, it is symmetric, and its closure has domain 
\[
    \cD_{\rm min}^\sharp := H^2_0(\R^+, \C^n).
\]
The adjoint of $(h^\sharp_{\rm min}, \cD_{\rm min}^\sharp)$ is the operator $(h^\sharp_{\rm max}, \cD_{\rm max}^\sharp)$ where $h_{\rm max}^\sharp := - \partial_{xx}^2 + V(x)$ has domain
\[
    \cD_{\rm max}^\sharp := H^2(\R^+, \C^n).
\]
We have $\cD_{\rm min}^\sharp \subsetneq \cD_{\rm max}^\sharp$, so $h^\sharp$ is not essentially self-adjoint. This reflects the fact that some boundary conditions must be chosen at $x = 0$.
The particularity of second order elliptic operators comes from the second Green's identity. 

\begin{lemma}[second Green's identity]\label{lem:second_Green}
    For all $\phi, \psi \in \cD_{\rm max}^\sharp$,
    \[
        \bra \phi, h_{\rm max}^\sharp \psi \ket_{\cH^\sharp} -  \bra h_{\rm max}^\sharp \phi,  \psi \ket_{\cH^\sharp} = \bra \phi(0), \psi'(0) \ket_{\C^n} - \bra \phi'(0), \psi(0) \ket_{\C^n}.
    \]
\end{lemma}

This suggests to introduce the boundary space
\[
\cH_b := \C^n \times \C^n
\] 
with its canonical symplectic form $\omega$ defined in Example~\ref{ex:symplectic_C2n}. We also introduce the map $\Tr : \cD_{\rm max}^\sharp \to \cH_b$ defined by
\begin{equation} \label{eq:def:Tr}
    \forall \phi \in \cD_{\rm max}^\sharp, \quad \Tr(\phi) := (\phi(0), \phi'(0)) \in \cH_b.
\end{equation}
With these notations, the second Green's identity reads
\[
    \forall \phi, \psi \in \cD^\sharp_{\rm max}, \quad \bra \phi, h_{\rm max}^\sharp \psi \ket_{\cH^\sharp} -  \bra h_{\rm max}^\sharp \phi,  \psi \ket_{\cH^\sharp} =
    \omega \left( \Tr(\phi), \Tr(\psi) \right).
\]

\medskip

We denote by $\| \cdot \|_\sharp$ the graph norm of $h^\sharp$, that is
\[
    \forall \phi \in \cD_{\rm max}^\sharp, \quad \| \phi \|_{\sharp}^2 := \| \phi \|_{\cH^\sharp}^2 + \| h^\sharp_{\rm max} \phi \|_{\cH^\sharp}^2.
\]
In the one-dimensional Hill setting, the graph norm is equivalent to the $H^2$-norm. Recall that a closed extension of $h^\sharp$ has a domain which is closed for this norm.

\begin{lemma} \label{lem:Tr_is_onto_finite_dim}
    The map $\Tr : (\cD_{\rm max}^\sharp, \| \cdot \|_{\sharp}) \to \cH_b$ is well-defined, continuous and onto.
\end{lemma}

\begin{proof}
    Since $V$ is bounded, the graph norm  $\| \cdot  \|_{\sharp}$ is equivalent to the usual $H^2(\R^+, \C^n)$ norm on $\cD^\sharp_{\rm max} = H^2(\R^+, \C^n)$. Rellich embedding shows that $H^2(\R^+, \C^n) \hookrightarrow C^1([0, \infty), \C^n)$ with continuous embedding. This implies that $\Tr$ is a bounded linear operator. Let $C, S \in C^\infty(\R^+, \R)$ be two compactly supported smooth functions with $C(0) = S'(0) = 1$ and $C'(0) = S(0) = 0$. Given an element $(u, u') \in \cH_b$, we have $(u, u') = \Tr(\psi)$ for $\psi(x) := u C(x) + u' S(x) \in \cD_{\rm max}^\sharp$, so $\Tr$ is onto.
\end{proof}

The next result shows that the self-adjoint extensions of $h^\sharp$ can be seen as Lagrangian planes of $\cH_b$.

\begin{theorem}\label{th:self_adjoint_extensions_are_Lagrangian_planes}
    Let $\cD^\sharp$ be a domain satisfying $\cD_{\rm min}^\sharp \subset \cD^\sharp \subset \cD_{\rm max}^\sharp$, and let $\ell := \Tr( \cD^\sharp)$. The adjoint domain of $(h^\sharp_{\rm max}, \cD^\sharp)$ satisfies $(\cD^\sharp)^* = \Tr^{-1} \left(\ell^\circ \right)$. \\
    In particular, $(h^\sharp_{\rm max}, \cD^\sharp)$ is a self-adjoint extension of $h^\sharp$ iff 
    \[
        \exists \ell \in \Lambda(\cH_b) \quad \text{so that} \quad \cD^\sharp = \Tr^{-1} (\ell).
    \]
\end{theorem}

\begin{proof}
    Since $\cD_{\rm min}^\sharp \subset \cD^\sharp \subset \cD_{\rm max}^\sharp$ and $(\cD_{\rm max}^\sharp)^* = \cD_{\rm min}^\sharp$, we have $\cD_{\rm min}^\sharp \subset \left(\cD^\sharp \right)^* \subset \cD_{\rm max}^\sharp$ as well.
    Let $\psi_0 \in (\cD^\sharp)^* \subset \cD_{\rm max}^\sharp$ . By definition of the adjoint, and the second Green's identity, we have
    \[
    \forall \phi \in \cD^\sharp, \quad 0 = \bra \psi_0, h^\sharp_{\rm max} \phi \ket_{\cH^\sharp} - \bra h^\sharp_{\rm max} \psi_0, \phi \ket_{\cH^\sharp} = \omega (\Tr(\psi_0),  \Tr(\phi)).
    \]
    We deduce that $\Tr(\psi_0) \in \ell^\circ$.  So $\Tr((\cD^\sharp)^*) \subset \ell^\circ$, which implies $(\cD^\sharp)^* \subset \Tr^{-1} (\ell^\circ)$.
    
    Conversely, let $\psi_0 \in \Tr^{-1}(\ell^\circ)$. By definition of $\ell^\circ$ and the second Green's identity, we get
    \[
    \forall \phi \in \cD^\sharp, \quad 
    0 = \omega \left( \Tr(\psi_0), \Tr(\phi) \right) = \bra \psi_0, h_{\rm max}^\sharp \phi \ket_{\cH^\sharp} - \bra h_{\rm max}^\sharp \psi_0, \phi \ket_{\cH^\sharp}.
    \]
    In particular, the map $T_{\psi_0} : \cD^\sharp \to \C$ defined by 
    \[
    T_{\psi_0} : \phi \mapsto \bra \psi_0, h_{\rm max}^\sharp \phi \ket_{\cH^\sharp} = \bra h_{\rm max}^\sharp \psi_0, \phi \ket_{\cH^\sharp}
    \]
    is bounded on $\cD^\sharp$ with $\| T_{\psi_0} \phi \|_{\cH^\sharp} \le \| h^\sharp_{\rm max} \psi_0 \|_{\cH^\sharp} \| \phi \|_{\cH^\sharp}$. So $\psi_0$ is in the adjoint domain $(\cD^\sharp)^*$. This proves as wanted that $\Tr^{-1}(\ell^\circ) \subset (\cD^\sharp)^*$, and finally $\Tr^{-1}(\ell^\circ) = (\cD^\sharp)^*$.
    
    Since $\Tr$ is onto, we have $\Tr \left( \Tr^{-1} (A) \right)$ for all $A \subset \cH_b$. On the other hand, if $\cD^\sharp$ defines a self-adjoint extension, then we have
    \[
    \cD^\sharp = \Tr^{-1} (\ell^\circ), \quad \text{with} \quad \ell := \Tr(\cD^\sharp).
    \]
    We deduce that $\ell = \Tr(\cD^\sharp) = \Tr \left( \Tr^{-1} (\ell^\circ) \right) = \ell^\circ$, hence $\ell$ is Lagrangian. Conversely, if $\ell$ is Lagrangian, we can define the domain $\cD^\sharp := \Tr^{-1}(\ell)$. We then have $\Tr(\cD^\sharp) = \Tr \left(  \Tr^{-1}(\ell) \right) = \ell$ by surjectivity of $\Tr$ again. In particular, the dual domain satisfies $(\cD^\sharp)^* = \Tr^{-1}(\ell^\circ) = \Tr^{-1}(\ell)  = \cD^\sharp$, so $(H^\sharp_{\rm max}, \cD^\sharp)$ is a self-adjoint extension. This concludes the proof.
\end{proof}

In what follows, we denote by $\left( h^\sharp, \ell^\sharp \right)$ the self-adjoint extensions of $h^\sharp$ with domain $\Tr^{-1}(\ell^\sharp)$.

Before we go on, let us give some examples of Lagrangian planes and their corresponding unitaries $\cU$ for some usual self-adjoint extensions. In the Hill's case, we have $\cH_b = \cH_1 \times \cH_2$ with $\cH_1 = \cH_2 = \C^n$, with the canonical symplectic form. In particular, the unitary $V$ in Assumption B is $V = \bbI_n$, and the unitaries $\cU : \C^n \to \C^n$ can be seen as elements of $\U(n)$.

\begin{example}[Dirichlet and Neumann boundary conditions] \label{example:Dirichlet}
    The Dirichlet extension of $h^\sharp$ corresponds to the Lagrangian plane $\ell_D := \{0 \} \times \C^n$, and the Neumann one corresponds to $\ell_N := \C^n \times \{ 0 \}$. To identify the corresponding unitary, we note that $(0, u') \in \ell_D$ can be written as
    \[
        \begin{pmatrix} 0 \\ u' \end{pmatrix}
        = 
        \begin{pmatrix} 1 \\ \ri \end{pmatrix} (-  \tfrac{\ri}2 u')
        +
        \begin{pmatrix} 1 \\ -\ri \end{pmatrix} ( \tfrac{\ri}2 u').
    \]
    Comparing with Lemma~\ref{lem:Lagrangian_unitaries_cU}, this gives the unitary $\cU_D := - \bbI_n \in \U(n)$. The proof for Neumann boundary conditions is similar, and we find $\cU_N := \bbI_n \in \U(n)$.
\end{example}

\begin{example}[Robin boundary conditions] \label{ex:Robin}
    Consider $\Theta$ and $\Pi$ two hermitian $n \times n$ matrices so that
    \[
        \Theta^*  = \Theta, \ \Pi^* =  \Pi, \ \Theta \Pi = \Pi \Theta, \quad \Theta^2 + \Pi^2 \quad \text{is invertible}.
    \]
    Let $\ell_{\Theta, \Pi}$ be the subspace
    \[
        \ell_{\Theta, \Pi} := \left\{ (\Theta x, \Pi x), \ x \in \C^n \right\} \subset \cH_b.
    \]
    We claim that $\ell_{\Theta, \Pi}$ is Lagrangian. Indeed, first we have
    \[
        \omega( (\Theta x, \Pi x), (\Theta y, \Pi y) ) =
        \bra \Theta x, \Pi y \ket_{\C^n} - \bra \Pi x, \Theta y \ket_{\C^n} = 
        \bra x, (\Theta \Pi - \Pi \Theta) y \ket = 0,
    \]
    so $\ell_{\Theta, \Pi} \subset \ell_{\Theta, \Pi}^\circ$. On the other hand, let $(z, z') \in \ell_{\Theta, \Pi}^\circ$. We have
    \[
       \forall x \in \C^n, \quad \bra z, \Pi x \ket_{\C^n} = \bra z', \Theta x \ket_{\C^n}, \quad \text{so} \quad
        \bra \Pi z - \Theta z', x \ket_{\C^n} = 0.
    \]
    We deduce that $\Pi z = \Theta z'$. In particular, setting $ z_0 = (\Theta^2 + \Pi^2)^{-1}  ( \Theta z + \Pi z')$, we have $z = \Theta z_0$ and $z' = \Pi z_0$, so $(z, z') = (\Theta z_0, \Pi z_0) \in \ell_{\Theta, \Pi}$. This proves that $\ell_{\Theta, \Pi}$ is Lagrangian. This also proves that
    \[
        \ell_{\Theta, \Pi} = \left\{ (z,z') \in \C^n \times \C^n, \quad \Pi z = \Theta z' \right\}.
    \]
    We say that the corresponding self-adjoint extension has the $(\Theta, \Pi)$-Robin boundary condition, namely $\psi \in H^2(\R^+, \C^n)$ is in the domain if
    \[
        \Pi \psi(0) =  \Theta \psi'(0).
    \]
    To identify the corresponding unitary, we remark that
    \[
        \begin{pmatrix} \Theta x \\ \Pi x \end{pmatrix}
        = 
        \begin{pmatrix}  1 \\ \ri \end{pmatrix} \tfrac12 \left( \Theta- \ri \Pi \right) x
        +
         \begin{pmatrix}  1 \\ -\ri \end{pmatrix} \tfrac12 \left( \Theta + \ri \Pi \right) x.
    \]
    Comparing with Lemma~\ref{lem:Lagrangian_unitaries_cU}, we recognise the unitary
    \[
        \cU_{\Theta, \Pi} := (\Theta + \ri \Pi) (\Theta - \ri \Pi)^{-1} \quad \in \U(n).
    \]
    Note that $A := (\Theta - \ri \Pi)$ is invertible, since $A^* A = \Theta^2 + \Pi^2$ is invertible. We recover Dirichlet boundary condition with the pair $(\Theta, \Pi) = (0, \bbI_n)$ and Neumann boundary condition with $(\Theta, \Pi) = (\bbI_n, 0)$.
\end{example}


\subsection{The Lagrangian planes $\ell^\pm(E)$}

In the previous section, we linked the boundary conditions at $x = 0$ with the Lagrangian planes of the boundary space $\cH_b$. We now focus on the Cauchy solutions of $H \psi = E \psi$. Since we are also interested in the behaviour at $-\infty$, we introduce $\cH^{\sharp, \pm} := L^2(\R^\pm)$ and the maximal domains
\[
    \cD^{\sharp, \pm}_{\rm max}:= H^2(\R^\pm, \C^n).
\]
The space $\cD^\sharp_{\rm max}$ considered previously corresponds to $\cD^{\sharp, +}_{\rm max}$. We also denote by $\Tr^\pm : \cD^{\sharp, \pm}_{\rm max} \to \cH_b$ the corresponding boundary trace operator 
\[
    \forall \psi \in \cD^{\sharp, \pm}_{\rm max}, \quad \Tr^\pm(\psi) = (\psi(0), \psi'(0)).
\]
Note that, due to the orientation of the line $\R$, the second Green's identity on the left-side reads
\begin{equation} \label{eq:left_Green}
    \forall \phi, \psi \in \cD^{\sharp, -}_{\rm max}, \quad
    \bra \phi, h_{\rm max}^{\sharp, -} \psi \ket_{\cH^{\sharp, -}} -  
    \bra h_{\rm max}^{\sharp, -}  \phi, \psi \ket_{\cH^{\sharp, -}}
    = - \omega \left( \Tr^-(\phi), \Tr^-(\psi) \right).
\end{equation}

We now set
\[
    \cS^\pm(E) := \Ker \left( h_{\rm max}^{\sharp, \pm} - E \right) = \left\{ \psi \in \cD^{\sharp, \pm}_{\rm max}, \quad -\psi'' + V \psi = E \psi \right\},
\]
and
\begin{equation} \label{eq:def:ell^pm(E)}
    \ell^\pm(E) := \left\{ \Tr^\pm(\psi), \quad \psi \in \cS^\pm(E)   \right\} \quad \subset \cH_b.
\end{equation}
The solutions in $\cS^\pm(E)$ can be seen as the set of Cauchy solutions which are square integrable at $\pm \infty$. Thanks to Cauchy's theory for ODEs, elements $\psi^\pm$ of $\cS^\pm(E)$ can be reconstructed from their boundary values $\Tr^\pm \left( \psi^\pm \right) \in \cH_b$. 


\begin{lemma} \label{lem:E_eigenvalue_of_h}
    For the bulk operator $h$, we have that for all $E \in \R$,
    \[
        \dim \Ker \left(h - E\right) = \dim \left( \ell^+(E) \cap \ell^-(E) \right).
    \]
    In particular, $E$ is an eigenvalue of $h$ iff $\ell^+(E) \cap \ell^-(E) \neq \{ 0 \}$.
\end{lemma}

\begin{proof}
    We will provide a general proof later (see the proof of Lemma~\ref{lem:E_eigenvalue_of_H}), which works in the Schrödinger case. Let us give a short proof using Cauchy's theory.
    
    Let $(u, u') \in \ell^+(E) \cap \ell^-(E)$, and let $\psi$ be the Cauchy solution of $- \psi'' + V \psi = E \psi$ with $\psi(0) = u$ and $\psi'(0) = u'$. By uniqueness of the Cauchy solution, the restriction of $\psi$ on $\R^\pm$ is in $\cS^{\pm}(E)$. In particular, $\psi$ is square integrable in $\pm \infty$, so $\psi \in L^2(\R, \C^n)$. Then, since $V$ is bounded, $\psi'' = (E - V) \psi$ is also in $L^2(\R, \C^n)$, and $\psi$ is in the domain  $H^2(\R, \C^n)$. As it satisfies $(h - E) \psi = 0$, it is an eigenvector of $h$ for the eigenvalue $E$. Conversely, if $\psi$ is such an eigenvector, then $\Tr( \psi) \in \ell^+(E) \cap \ell^-(E)$.
\end{proof}

One can therefore detect eigenvalues as the crossings of $\ell^+(E)$ and $\ell^-(E)$. We now prove that, when $E$ is in the resolvent set of the bulk operator, we have instead $\ell^+(E) \oplus \ell^-(E) = \cH_b$. Our proof only uses the fact that the bulk operator $h$ is essentially self-adjoint. 

\begin{theorem} \label{th:Lagrangian_planes_Hill_lE}
    For all $E \in \R \setminus \sigma(h)$, the sets $\ell^\pm(E)$ are Lagrangian planes of $\cH_b$, and
    \[
        \cH_b = \ell^+(E) \oplus \ell^-(E).
    \]    
\end{theorem}

This shows for instance that there are as many Cauchy's solutions which decay to $+ \infty$ as solutions which decay to $-\infty$ (here, they both form subspaces of dimension $n$). This is somehow reminiscent of the Weyl's criterion~\cite{weyl1910gewohnliche} (see also~\cite{lidskii1954number}).

Again, we postpone the proof to the Schrödinger section (see the proof of Theorem~\ref{th:Lagrangian_planes_Schrodinger_lE} below), as it is similar, but somehow looks more complex in the PDE setting.

\begin{remark} \label{rem:ell+_and_ell-}
    In the proof given below, we use that $h = - \partial_{xx}^2 + V$ is self-adjoint on the whole line, and deduce that $\ell^+(E)$ and $\ell^-(E)$ are both Lagrangian planes. Note however that $\ell^+(E)$ is independent of $V$ on $\R^-$. So $\ell^+(E)$ is a Lagrangian plane whenever there exists an extension of $V$ on $\R^-$ for which the corresponding bulk operator has $E$ in its resolvent set.
\end{remark}

\begin{remark} \label{rem:not-always-lagrangian-planes}
    The spaces $\ell^\pm(E)$ are not always Lagrangian planes. For instance, if $V : \R \to \R$ is $1$-periodic, the spectrum of $h := - \partial_{xx}^2 + V$ is composed of bands and gaps. For $E \in \sigma(h)$, the set of solutions of $(h - E) \psi$ is two-dimensional, and spanned by two quasi-periodic functions, hence the solutions never decay at $\pm \infty$. So for all $E \in \sigma(h)$, we have $\ell^+(E) = \ell^-(E) = \emptyset$.
\end{remark}


At this point, we defined two types of Lagrangian planes for a given operator $h$. First, we defined the planes $\ell^\sharp$ representing the boundary conditions of a self-adjoint extension of the edge Hamiltonian $h^\sharp$. Then, we defined the planes $\ell^+(E)$ as the set of traces of $ \Ker (h_{\rm max}^\sharp - E) $. If $\Tr(\psi) \in \ell^+(E) \cap \ell^\sharp$, then $\psi$ is in the domain of $h^\sharp$, and satisfies $(h^\sharp - E) \psi = 0$. So $\psi$ is an eigenvector for the eigenvalue $E$. This proves the following (compare with Lemma~\ref{lem:E_eigenvalue_of_h}).

\begin{lemma} \label{lem:dimKer=dimCap}
    Let $E \in \R \setminus \sigma(h)$, and consider a self-adjoint extension $(h^\sharp, \ell^\sharp)$ of the edge operator. Then
    \[
        \dim \Ker \left( h^\sharp - E  \right) = \dim \left(\ell^{+}(E) \cap  \ell^\sharp  \right).
    \]
\end{lemma}
This result is of particular importance, since we detect eigenvalues as the crossing of two Lagrangian planes. The first one $\ell^+(E)$ only depends on bulk properties ({\em e.g.} on the potential $V$), while the second one $\ell^\sharp$ only depends on the chosen boundary conditions at the edge (and is usually independent of the choice of $V$). 

If in addition Assumption B holds, then we can introduce $\cU^+(E)$ and $\cU^\sharp$ the unitaries corresponding to the Lagrangian planes $\ell^+(E)$ and $\ell^\sharp$ respectively, and we have
\[
\dim \Ker \left( h^\sharp - E  \right) = \dim \left(\ell^{+}(E) \cap  \ell^\sharp  \right)
= \dim\left( (\cU^\sharp)^* \cU^{+}(E)  - 1 \right).
\]

\begin{remark}[Scattering coefficients] \label{rem:scattering}
    Let us give an interpretation of the unitary $\cU^+(E)$. For $E \notin \sigma(h)$ and $E > 0$, waves cannot propagate in the medium at energy $E$. Considering the half-medium $- \partial_{xx}^2 + \1(x > 0) V(x)$, any incident wave coming from the left of the form $\re^{ \ri k x} u$ with $k := \sqrt{E}$ and $u \in \C^n$ ``must bounce back". According to scattering theory, there is a unitary operator $R(k) \in \U(n)$, called the {\em reflection coefficient}, so that there is a continuous solution $\psi_u(x)$ of
    \[
        \begin{cases}
            \psi_u(x)  = \re^{ \ri k x} u + \re^{ - \ri k x} \left( R(k) u \right) & \quad \text{for} \quad  x < 0 \\
            \left( - \partial_{xx}^2 + V(x) \right) \psi_u(x) = 0 & \quad \text{for} \quad  x > 0
        \end{cases},
    \]
    and $\psi_u(x)$ is square-integrable at $+ \infty$ (no transmission). This shows that $\psi_u \1_{\R^+} \in \cS^+(E)$. Taking boundary values, we obtain
    \[
       \forall u \in \C^n, \quad  \begin{pmatrix}
            \left( \bbI_n + R(k) \right) u \\
            \ri k \left( \bbI_n - R(k) \right) u
        \end{pmatrix} \in \ell^+(E) = \left\{
     \begin{pmatrix}
         1 \\ \ri 
     \end{pmatrix} x + \begin{pmatrix}
     1 \\ -\ri 
 \end{pmatrix} \cU^+(E) x, \quad x \in \C^n \right\}.
    \]
    This leads to the identity
    \[
         k \cdot \dfrac{1 - R(k)}{1 + R(k)} = \frac{1 - \cU^+(E)}{1 + \cU^+(E)}, \quad k = \sqrt{E}.
    \]
    In other words, the Cayley transform of $\cU^+(E)$ equals the one of the reflection coefficient $R(k)$, up to the multiplicative factor $k$. 
\end{remark}


\section{Families of Hill's operators}
\label{sec:family_Hill}

In the previous section, we exhibit the relationships between self-adjoint extensions, Lagrangian planes, and unitaries. We now consider periodic families of these objects, parametrised by $t \in \TT^1$, namely $h^\sharp_t$, $\ell_t^\sharp$ and $\cU_t^\sharp$. For each such family, we define an index, namely a {\em spectral flow} across $E$ for the family $h^\sharp_t$, a {\em Maslov index} for the bifamily $(\ell_t^+(E), \ell_t^\sharp)$ and a {\em spectral flow} across $1$ for the family $(\cU_t^\sharp )^* \cU_t^+(E)$. All these objects are defined in the following sections, and we prove that they all coincide.

All these indices can be defined for {\em continuous} families. However, since the proofs are simpler in the continuously differentiable case, we restrict ourselves to this case. As these indices depend only on the homotopy class of the corresponding loops, similar results hold in the continuous case.


\subsection{Families of Lagrangians, and Maslov index}
\label{ssec:Maslov}

We first define the Maslov index of two families of Lagrangian spaces. This index originates from the work of Maslov in~\cite{maslov1972theorie, arnol1967characteristic}. In these works, the index was defined for finite dimensional real symplectic spaces (namely $\R^{2n}$ in Example~\ref{ex:symplectic_R2n}). A modern approach can be found in~\cite{furutani2004fredholm}, where the infinite dimensional case is studied as well. Here, we present a simple version of the theory, which is enough for our purpose.

Let $(\cH_b, \omega)$ be a symplectic Hilbert space (not necessarily finite dimensional). We define a topology on the Lagrangian Grassmanian $\Lambda(\cH_b)$ by setting
\[
    \forall \ell_1, \ell_2 \in \Lambda(\cH_b), \quad 
    {\rm dist}(\ell_1, \ell_2) := \| P_1 - P_2 \|_{\rm op},
\]
where $P_1$ and $P_2$ are the orthogonal projectors on $\ell_1$ and $\ell_2$ respectively. A family $\ell(t)$ in $\Lambda(\cH_b)$ is said to be continuous, continuously differentiable, {\em etc.} if the corresponding family of projectors $P(t)$ is so in $\cB\left(\cH_b\right)$.

\subsubsection{Definition with quadratic crossing forms}

Consider two continuously differentiable families $\TT^1 \ni t \mapsto \ell_1(t)$ and $\TT^1 \ni t \mapsto \ell_2(t)$. Let $t^* \in \TT^1$ be such that $\ell_1(t^*) \cap \ell_2(t^*) \neq \{ 0 \}$. We define the sesquilinear form $b$ on $\ell_1(t^*) \cap \ell_2(t^*)$ by
\begin{equation} \label{eq:def:sesquilinearForm}
\forall x,y \in \ell_1(t^*) \cap \ell_2(t^*) , 
\quad b_{\ell_1, \ell_2}(x,y) := \omega( x, P_1'(t^*) y ) - \omega( x, P_2'(t^*) y ).
\end{equation}

\begin{lemma}
    The sesquilinear form $b_{\ell_1, \ell_2}$ is hermitian: $b_{\ell_1, \ell_2}(x,y) = \overline{b_{\ell_1, \ell_2}(y,x)}$.
\end{lemma}
\begin{proof}
    Let $P_t := P_1(t)$. First, since ${\rm Ran} \, P_1(t)  = \ell_1(t)$ is isotropic for all $t$, we have
    \[
    \forall x, y \in \cH_b, \quad \forall t \in \TT^1, \quad \omega(P_t (x), P_t(y)) = 0.
    \]
    Differentiating gives
    \[
    \omega(P_t (x), P'_t (y) ) = - \omega(P'_t (x), P_t(y))  = \overline{\omega(P_t (y), P'_t (x) )}.
    \]
    Taking $t = t^*$ and $x, y \in \ell_1(t^*) \cap \ell_2(t^*)$, so that $P_{t^*}(x) = x$ and $P_{t^*}(y) = y$ gives 
    \[
        \forall x, y \in \ell_1(t^*) \cap \ell_2(t^*), \quad 
        \omega(x, P_t'(y)) = \overline{\omega(y, P_t'(x))}.
    \]
    A similar equality holds for $P_t = P_2(t)$, which proves that $b_{\ell_1, \ell_2}$ is hermitian.
\end{proof}
In particular, all eigenvalues of $b_{\ell_1, \ell_2}$ are real-valued. We say that $t^*$ is a {\bf regular crossing} if $\ell_1(t^*) \cap \ell_2(t^*)$ is finite dimensional (say of dimension $k \in \N$), and if all eigenvalues $(\mu_1, \cdots, \mu_k)$ of $b$ are non null (so the corresponding quadratic form is non degenerate). For such crossings, we set
\[
    {\rm deg}(t^*) =  \sum_{j=1}^k {\rm sgn} \left( \mu_j \right).
\]
The pair $(\ell_1(t), \ell_2(t))$ is regular if all crossings are regular. For such pair, the Maslov index is defined by
\[
    \boxed{ \Mas \left(\ell_1, \ell_2, \TT^1  \right) := \sum_{t^* \ \text{regular crossing}} {\rm deg} (t^*)
        \quad \in \Z. }
\]
It is clear from the definition that $\Mas(\ell_1, \ell_2, \TT^1) = - \Mas(\ell_2, \ell_1, \TT^1)$. This definition does not require Assumptions A (nor B).

\subsubsection{Definition with the unitaries $U$}

In the case where Assumption $A$ holds, we can relate the Maslov index to a spectral flow. Consider two continuously differentiable loops of Lagrangian $\ell_1(t)$ and $\ell_2(t)$ from $t \in \TT^1$ to $\Lambda(\cH_b)$. Let $U_1(t)$ and $U_2(t)$ be the corresponding unitaries from $\Ker(J - \ri)$ to $\Ker(J + \ri)$. Then $U_1$ and $U_2$ are continuously differentiable for the operator norm topology. From Lemma~\ref{lem:KerU1U2-1}, we have that for all $t \in \TT^1$,
\[
    \dim \Ker \left( U_2(t)^* U_1(t) - \bbI_{\Ker(J - \ri)}  \right) = \dim \left( \ell_1(t) \cap \ell_2(t) \right).
\]
In particular, if all crossings are regular, then $\dim(\ell_1 \cap \ell_2) = \Ker(U_2^* U_1 - 1)$ is finite dimensional. Let $t^* \in \TT^1$ be such that the kernel is non empty, of dimension $k \in \N$. By usual perturbation theory for operators~\cite{kato2013perturbation}, there are $k$ continuously differentiable branches of eigenvalues of the unitary $U_2^* U_1$ crossing $1$ around $t^*$. More specifically, we have the following.

\begin{lemma} \label{lem:branches_eigenvalues}
    Let $U(t)$ be a periodic continuously differentiable family of unitaries, and let $t^* \in \TT^1$ be such that
    \[
        \dim \Ker \left( U(t^*) - 1\right) =: k \in \N.
    \]
    Then, there is $\varepsilon > 0$, $\eta > 0$ and $k$ continuously differentiable functions $\left\{ \theta_1(t), \cdots, \theta_k(t) \right\}$ from $t \in (t^* - \varepsilon, t^* + \varepsilon)$ to $\SS^1 := \{ z \in \C, \ | z | = 1\}$, so that
    \[
        \sigma \left( U(t) \right) \cap B(1, \eta) = \{ \theta_1(t), \cdots, \theta_k(t) \} \cap B(1, \eta).
    \]
\end{lemma}
The functions $\theta_j$ are the branches of eigenvalues of $U$. We say that $t^*$ is a regular crossing if $k := \dim \Ker(U(t^*) - 1) < \infty$, and if $\theta_j'(t^*) \neq 0$ for all $1 \le j \le k$. Note that since $\theta_j$ has values in $\SS^1$, we have $\theta_j'(t^*) \in \ri \R$. The degree of $t^*$ is
\[
    {\rm deg}(t^*) := \sum_{j = 1}^k {\rm sgn} \left( - \ri \theta'_j  (t^*) \right).
\]
This is the net number of eigenvalues crossing $1$ in $\SS^1$ in the positive (counter-clockwise) direction. Finally, if all crossings are regular, the {\bf spectral flow} of $U$ across $1$ is
\[
    \boxed{ \Sf \left( U, 1, \TT^1  \right) := \sum_{t^* \ \text{regular crossing}} {\rm deg (t^*)} \quad \in \Z. }
\]

\begin{lemma}
    Let $\ell_1(t)$ and $\ell_2(t)$ be two continuously differentiable families of Lagrangians in $\Lambda(\cH_b)$, and let $U_1(t)$ and $U_2(t)$ be the corresponding unitaries. Then $t^* \in \TT^1$ is a regular crossing of $(\ell_1, \ell_2)$ iff it is a regular crossing of $U_2^* U_1$. If all crossings are regular, then,
    \[
        \Mas\left(\ell_1, \ell_2, \TT_1 \right) = \Sf \left(U_2^* U_1, 1, \TT^1 \right).
    \]
\end{lemma}

\begin{proof}
    For the sake of simplicity, we assume that only $\ell_1$ depends on $t$. The proof is similar in the general case. Let $t^*$ be a regular crossing point, and let 
    \[
        k := \dim \left( \ell_1(t^*) \cap \ell_2 \right) = \dim \Ker (U_2^* U_1(t^*) - 1).
     \] 
    Let $\theta_1, \theta_2, \cdots \theta_k$ be the branches of eigenvalues crossing $1$ at $t = t^*$ (see Lemma~\ref{lem:branches_eigenvalues}), and let $x_1^-(t), \cdots, x_k^-(t)$ be a corresponding continuously differentiable set of orthonormal eigenfunctions in $\Ker(J - \ri)$. First, we have, for all $1 \le i, j \le k$, and all $t \in (t^* - \varepsilon, t^* + \varepsilon)$,
    \[
         \bra x_i^-, \left[ U_2^* U_1(t)  - \theta_j \right] x_j^-  \ket_{\cH_b} = 0.
    \]
    Differentiating and evaluating at $t = t^*$ shows that
    \[
    \bra  x_i^-, \partial_t \left[ U_2^{*} U_1  - \theta_j \right] x_j^-  \ket_{\cH_b}
    +  \bra \left[ U_2^{*} U_1(t^*)  - 1 \right]^* x_i^-, (\partial_t x_j^-)  \ket_{\cH_b} = 0.
    \]
    At $t = t^*$, we have $U_2^{*} U_1(t^*) x_i^- = x_i^-$, so $U_1^{*}(t^*) U_2 x_i^- = x_i^-$ as well, and the last term vanishes. We get the Hellmann-Feynman equation
    \[
    \delta_{ij} \theta_j'(t^*) 
    = \bra U_2 x_i^-,  (\partial_t U_1) x_j^- \ket_{\cH_b} \big|_{t = t^*}
    = \bra U_1 x_i^-,  (\partial_t U_1) x_j^- \ket_{\cH_b} \big|_{t = t^*}.
    \]
    
    On the other hand, we set 
    \begin{equation} \label{eq:def_xj}
        x_j(t) := x_j^-(t) + U_1(t) x_j^-(t) \quad \in \ell_1.
     \end{equation}
     We have $x_j \in \ell_1$ for all $t$, so $P_1 x_j = x_j$ for all $t$. Differentiating gives
    \[
        \left(\partial_t P_1 \right) x_j + P_1 \left(\partial_t x_j\right) = (\partial_t x_j).
    \]
    Since $P_1 \left( \partial_t x_j\right) \in \ell_1$, which is Lagrangian, we have $\omega(x_i, P_1( \partial_t x_j)) = 0$, so
    \begin{equation} \label{eq:omega(x,P'x)}
    \omega \left(x_i,  \left(\partial_t P_1 \right) x_j \right) = \omega(x_i, \partial_t x_j).
    \end{equation}
    In addition, differentiating~\eqref{eq:def_xj} shows that
    \[
    \partial_t x_j = \left[ 1 + U_1 \right] (\partial_t x_j^-) + \left( \partial_t U_1\right) x_j^-.
    \]
    Since $x_j^- \in \Ker (J - \ri)$  for all $t$, we have $(\partial_t x_j^-) \in \Ker(J - \ri)$ as well, and the first term is in $\ell_1$. On the other hand, $(\partial_t U_1) x_j^-$ is in $\Ker(J + \ri)$. Combining with~\eqref{eq:omega(x,P'x)} and~\eqref{eq:cross_terms_omega}, this gives
    \[
        \omega \left(x_i,  \left(\partial_t P_1 \right) x_j \right) = \omega(x_i, \partial_t x_j) = \omega (x_i, \left( \partial_t U_1\right) x_j^-)
     = \omega( U_1 x_i^-, \left( \partial_t U_1 \right) x_j^- ).
    \]
    Using that $\omega(x,y) = \bra x, Jy \ket_{\cH_b}$ and that $(\partial_t U_1) x_j^- \in \Ker(J + \ri)$, we obtain, at $t = t^*$, and recalling the definition of $b$ in~\eqref{eq:def:sesquilinearForm},
    \begin{align*}
         b \left(x_i, x_j \right) & = \omega(x_i, \left( \partial_t P_1 \right) x_j) =  \omega( U_1 x_i^-, \left( \partial_t U_1 \right) x_j^- ) \\
         & = \left\bra U_1 x_i^-, J \left( \partial_t U_1 \right) x_j^- \right\ket
          = -\ri \bra U_1 x_i^-, (\partial_t U_1) x_j^- \ket_{\cH_b} = \delta_{ij} (-\ri) \theta'_j(t^*).
    \end{align*}
   The sesquilinear form $b$ is therefore diagonal in the $(x_1, \cdots, x_k)$ basis, with corresponding eigenvalues $(-\ri \theta_j')$. That concludes the proof.
\end{proof}


\subsubsection{Definition with the unitary $\cU$.} 
In the case where the stronger Assumption B holds, one has a similar result with the unitaries $\cU$ instead of $U$. We state it without proof, as it is similar to the previous one.
\begin{lemma} \label{lem:Maslovl1l2_with_cU}
    If $(\cH_b = \cH_1 \times \cH_2, \omega)$ satisfies Assumption~B.
    Let $\ell_1(t)$ and $\ell_2(t)$ be two continuously differentiable families of Lagrangian planes in $\Lambda(\cH_b)$, and let $\cU_1(t)$ and $\cU_2(t)$ be the corresponding unitaries of $\cH_1$. Then $t^* \in \TT^1$ is a regular crossing of $(\ell_1, \ell_2)$ iff it is a regular crossing of $\cU_2^* \cU_1$. If all crossings are regular, then,
    \[
    \Mas\left(\ell_1, \ell_2, \TT_1 \right) = \Sf \left(\cU_2^* \cU_1, 1, \TT^1 \right).
    \]
\end{lemma}

The importance of this lemma comes from the fact that, in the finite dimensional case ($\cH_1 \approx \C^n$), the spectral flow of a periodic family $\cU(t) \in \U(n)$ across $1$ (or any other point in $\SS^1$) equals the winding number of $\det \cU(t)$:
\[
    \Sf \left( \cU, z \in \SS^1, \TT^1  \right) = \Winding \left( \det(\cU), \TT^1 \right).
\]
In our case with $\cU = \cU_2^* \cU_1$, we have $\det(\cU_2^* \cU_1) = \det(\cU_1)/ \det(\cU_2)$, hence 
\[
    \Winding \left( \det(\cU_2^* \cU_1), \TT^1 \right) = \Winding \left( \det \cU_1, \TT^1\right) -  \Winding \left(\det \cU_2, \TT^1\right),
\]
that is, the index splits. 
\begin{definition} \label{def:I}
    For a periodic family of (finite dimensional) Lagrangians $\ell(t)$ with corresponding unitaries $\cU(t)$, we define the index
    \begin{equation*} 
        \cI(\ell, \TT^1) := {\rm Winding}( \det (\cU(t)), \TT^1) \quad \in \ZZ.
    \end{equation*}
\end{definition}
We can reformulate Lemma~\ref{lem:Maslovl1l2_with_cU} as
\begin{equation} \label{eq:splitting}
        \Mas\left(\ell_1, \ell_2, \TT^1 \right) = \cI (\ell_1, \TT^1) - \cI (\ell_2, \TT^1).
\end{equation}

\subsection{Families of Hill's operators, spectral flow} 
\label{ssec:spectral_flow}

We now focus on a periodic family of Hill's operators $\left(h_t \right)_{t \in \TT^1}$. Let $\TT^1 \ni t \mapsto V_t$ be a periodic family of potentials satisfying~\eqref{eq:form_V}, and set
\[
    h_t := - \partial_{xx}^2 + V_t(x).
\]
We assume that $t \mapsto V_t$ is continuously differentiable as a map from $\TT^1$ to the Banach space $L^\infty(\R, \cS_n)$. Since $\TT^1$ is compact, $V(t,x)$ is uniformly bounded. In particular, as in Section~\ref{ssec:sa_Hill}, the operator $h_t$ is essentially self-adjoint with fixed domain $\cD= H^2(\R^+, \C^n)$ for all $t \in \TT^1$.

The {\bf spectrum} of the family $(h_t)_{t \in \TT^1}$ is the set
\[
    \sigma \left( h_t, \TT^1 \right) := \bigcup_{t \in \TT^1} \sigma ( h_t ).
\]
It is the compact union of all spectra of $(h_t)$ for $t \in \TT^1$. Since $t \mapsto \sigma(h_t)$ is continuous, $\sigma \left(h_t, \TT^1 \right)$ is a closed set in $\R$. The complement of $\sigma \left(h_t, \TT^1\right)$ is the {\bf resolvent set} of the family $(h_t)_{t \in \TT^1}$.


We now consider a corresponding family of edge self-adjoint operators, of the form $(h_t^\sharp, \ell_t^\sharp)$. We say that this family is continuous, continuously differentiable, {\em etc.} if the corresponding family of Lagrangian planes $\ell_t^\sharp$ is so in $\Lambda(\cH_b)$.

Fix $E \in \R$ in the resolvent set of $(h_t)_{t \in \TT^1}$. As $t$ varies in $\TT^1$, the spectrum of the bulk operator $h_t$ stays away from $E$. However, for the edge operators $(h^\sharp_t, \ell_t^\sharp)$, some eigenvalues may cross the energy $E$. If $t^* \in \TT^1$ is such that $\dim \Ker \left(h^\sharp_{t^*} - E \right) = k \in \N$, then, as in Lemma~\ref{lem:branches_eigenvalues}, we can find $\varepsilon > 0$, $\eta > 0$ and $k$ continuously differentiable branches of eigenvalues $\lambda_j(t) \in \R$ so that, for $t \in (t^* - \varepsilon, t^* + \varepsilon)$,
\[
    \sigma \left( h_t^\sharp  \right) \cap B(E, \eta) = \{ \lambda_1(t), \cdots,  \lambda_k (t)\} \cap B(E, \eta).
\]
At $t = t^*$, we have $\lambda_1(t^*) = \cdots = \lambda_k(t^*) = E$. The crossing $t^*$ is regular if $\lambda_j'(t^*) \neq 0$ for all $1 \le j \le k$. For such a crossing, we set
\[
    {\rm deg}(t^*) = \sum_{j=1}^k {\rm sgn} \left( \lambda_j'(t^*) \right).
\]
We say that the energy $E$ is a {\bf regular energy} if all crossings at $E$ are regular. For such an energy, we define the {\bf spectral flow} of $(h_t^\sharp)$ across $E$ as the net number of eigenvalues crossing $E$ downwards (see~\cite{atiyah_patodi_singer_1976, phillips1996self, waterstraat2016fredholm}):
\[
    \boxed{ \Sf( h^\sharp_t, E, \TT^1) := - \sum_{t^* \text{regular crossing}} {\rm deg}(t^*) \quad \in \Z.}
\]

The main result of this section is the following. Recall that the index $\cI$ was defined in Definition~\ref{def:I}, and that we consider operators on the right half-space.
\begin{theorem} \label{th:main-Hill}
    Let $(a,b) \subset \R$ be any interval in $\R \setminus \sigma \left( h_t, \TT_1 \right)$. Then
    \begin{itemize}
        \item almost any $E$ in $(a,b)$ is a regular energy for $(h_t^\sharp, \ell_t^\sharp)$; 
        \item for any regular energy $E$ in $(a, b)$, we have
        \[
            \Sf(h^\sharp_t, E, \TT^1) = \Mas \left( \ell_t^+(E), \ell_t^\sharp, \TT_1  \right) 
            = \cI(\ell_t^+(E), \TT^1) - \cI(\ell_t^\sharp, \TT^1).
        \]
    \end{itemize}
\end{theorem}

\begin{proof} 
    The first part comes from Sard's lemma, and can be proved as in~\cite[Lemma III.18]{Gontier2020edge}.
    
    Fix $E$ a regular energy, let $t^*$ be a crossing point so that $\dim \Ker \left( h_{t^*}^\sharp - E \right) = k \in \N$, and let $\lambda_1, \cdots,  \lambda_k$ be the corresponding branches of eigenvalues. The idea of the proof is to follow the two families of branches $(t, \lambda_j(t))$ and $(t, E)$, describing respectively $\ell^\sharp_t$ and $\ell^+_t(E)$. 
    
    For the first branch, let $\psi_1(t), \cdots, \psi_k(t)$ be a continuously differentiable family of $\cH^\sharp$-orthonormal eigenvectors in $\cD^\sharp_t := \Tr^{-1} \left(\ell_t^\sharp\right)$, so that 
    \begin{equation} \label{eq:hpsi=lambdapsi}
        h_t^\sharp \psi_j(t) = \lambda_j (t) \psi_j(t),
    \end{equation}
    and let 
    \[
    u_j := \Tr\left( \psi_j\right), \quad \text{so that} \quad
    \ell^+_{t^*}(E)  \cap \ell^\sharp_{t^*} = {\rm Span} \left\{ u_1(t^*), \cdots, u_k(t^*)  \right\}.
    \]
    For all $t \in (t^* - \varepsilon, t^* + \varepsilon)$, we have (recall that $h_t^\sharp \subset h_{t, \rm max}^\sharp$)
    \[
        \bra \psi_i(t), (h^\sharp_{t, {\rm max}} - \lambda_j(t)) \psi_j(t) \ket_{\cH^\sharp} = 0.
    \]
    Differentiating and evaluating at $t = t^*$ gives
    \begin{align*}
        &  \bra \left( \partial_t \psi_i \right), (h^\sharp_{t^*, {\rm max}} - E )  \psi_j  \ket_{\cH^\sharp} \big|_{t = t^*}
        +  \bra \psi_i, (h^\sharp_{t^*, {\rm max}} - E )  \left( \partial_t \psi_j \right) \ket_{\cH^\sharp} \big|_{t = t^*} \\
        & \quad +        \bra \psi_i, \partial_t(h^\sharp_{t, {\rm max}} - \lambda_j ) \psi_j \ket_{\cH^\sharp} \big|_{t = t^*}
         = 0
    \end{align*}
    The first term vanishes with~\eqref{eq:hpsi=lambdapsi}. For the second term, we put the operator $(h_{t^*, {\rm max}}^\sharp - E)$ on the other side using the second Green's identity, and we get
    \[
        \bra \psi_i, (h^\sharp_{t^*, {\rm max}} - E )  \left( \partial_t \psi_j \right) \ket_{\cH^\sharp} \big|_{t = t^*} = \omega \left( u_i, \partial_t u_j  \right) \big|_{t = t^*} =  \omega \left( u_i,  \left(\partial_t P^\sharp_{t^*} \right) u_j \right) \big|_{t = t^*}.
    \]
    For the last equality, we introduced $P^\sharp_t$ the projection on $\ell_t^\sharp$, and used an equality similar to~\eqref{eq:omega(x,P'x)}. This gives our first identity
    \[
         \delta_{ij} \lambda_j'(t^*) =    \bra \psi_i, \partial_t(h^\sharp_{t, {\rm max}}  ) \psi_j \ket_{\cH^\sharp} \big|_{t = t^*} + \omega \left( u_i,  \partial_t \left( P^\sharp_t \right) u_j \right) \big|_{t = t^*}.
    \]
    
    \begin{remark}
        In the case where the domain $\ell^\sharp_t = \ell^\sharp$ is independent of $t$, we recover the Hellmann-Feynman identity $\bra \psi_i, \partial_t(h^\sharp_{t} - \lambda_j ) \psi_j \ket_{\cH^\sharp} \big|_{t = t^*} = 0$.
    \end{remark}
    
    For the second branch, let $(\phi_1(t), \cdots, \phi_k(t))$ be a smooth family of linearly independent functions in $\cS^+_t(E)$, and so that, at $t = t^*$, $\phi_j(t^*)= \1_{\R^+} \psi_j(t^*)$. We set
    \[
        v_j = \Tr(\phi_j), \quad \text{so, at $t = t^*$}, \quad v_j(t^*)= u_j(t^*).
    \]
    This time, we have, for all $t \in (t^* - \varepsilon, t^* + \varepsilon)$,
    \[
        \left\bra \phi_i, \left( h^{\sharp}_{t, \rm max} - E \right)  \phi_j \right\ket = 0.
    \]
    Differentiating and evaluating at $t = t^*$ gives as before
    \[
         \left\bra \phi_i, \partial_t \left( h^{\sharp}_{t^*, \rm max} \right)  \phi_j \right\ket 
         = - \omega(v_i, \partial_t v_j) 
         = - \omega \left( v_i, \left( \partial_t P_{t}^+\right) v_j  \right) 
         = -\omega \left( u_i, \left( \partial_t P_{t}^+\right) u_j  \right).
    \]
    Gathering the two identities shows that
    \[
        \delta_{ij} \lambda_j'(t^*) =  \omega \left( u_i,  \partial_t \left( P^\sharp_t \right) u_j \right) \big|_{t = t^*} - \omega \left( u_i, \left( \partial_t P_t^+\right) u_j  \right) \big|_{t = t^*} .
    \]
    We recognise the sesquilinear form $b$ defined in~\eqref{eq:def:sesquilinearForm}. Actually, we proved that
    \[
        \delta_{ij} \lambda_j'(t^*) = b_{\ell_t^\sharp, \ell_t^+} \left( u_i,  u_j   \right).
    \] 
    This form is therefore diagonal in the $(u_1, \cdots, u_j)$ basis, and its eigenvalues are the $\lambda_j'(t^*)$. Counting the number of positive/negative $\lambda_j'(t^*)$, and summing over all regular crossings gives as wanted
    \begin{align*}
        \Sf \left( h_t^\sharp, E, \TT^1 \right) & = - {\rm Mas} \left( \ell_t^\sharp, \ell_t^+(E), \TT^1\right) = {\rm Mas} \left(\ell_t^+(E),  \ell_t^\sharp, \TT^1\right) \\
        & = \cI(\ell_t^+(E), \TT^1) - \cI(\ell_t^\sharp, \TT^1) .
    \end{align*}
\end{proof}

\begin{remark} \label{rem:left_Maslov}
    When considering the operators on the left half line, Green's formula has a minus sign (see Eqn.~\eqref{eq:left_Green}). The proof is therefore similar up to a sign change, and we get
    \[
         \Sf(h^{\sharp,-}_t, E, \TT^1) = - \Mas \left( \ell_t^-(E), \ell_t^\sharp, \TT_1  \right) 
        = \cI(\ell_t^\sharp, \TT^1)  - \cI(\ell_t^-(E), \TT^1) .
    \]
\end{remark}

\subsection{Bulk/edge index}

Theorem~\ref{th:main-Hill} states that the spectral flow of the edge operator $h^\sharp_t$ can be seen as the sum of two contributions: the quantity $\cI(\ell_t^+(E), \TT^1)$ which only depends on the bulk operator, and the quantity $\cI(\ell_t^\sharp, \TT^1)$ which only depends on the choice of boundary conditions, so on the nature of the edge. Should we choose the same boundary conditions for all operators $h^\sharp_t$, as it is usually the case, this spectral flow would only depend on the bulk quantity. 

So, although the spectral flow is related to edge modes, we emphasise that the index $\cI(\ell_t^+(E), \TT^1)$ really is a bulk quantity! This motivates the following definition.
\begin{definition}[Bulk/edge index] \label{def:bulk_edge}
    We define the {\em bulk/edge index} of the family of {\bf bulk} operators $(h_t)_{t \in \TT^1}$ at energy $E \notin \sigma(h_t)$ as the spectral flow of its (right) Dirichlet {\bf edge} restriction:
    \[
        \boxed{ \cI(h_t, E) := \Sf \left( h_{t, D}^{\sharp, +}, E, \TT^1 \right).}
    \]
\end{definition}
Note that we also have $\cI(h_t, E) = {\rm Mas} (\ell_t^+(E), \ell_D, \TT^1) = \cI(\ell_t^+(E), \TT^1)$ defined in Definition~\ref{def:I}. However, our definition of bulk/edge index does no rely on the notion of winding number, as was the case for $\cI(\ell_t^+(E), \TT^1)$. This definition will therefore works in the infinite dimensional PDE case, where there is no notion of winding number.

\begin{lemma} \label{lem:def:bulkedge_left}
    When considering the left Dirichlet edge restriction, we have
    \[
         \cI(h_t, E) = - \Sf \left( h_{t, D}^{\sharp, -}, E, \TT^1 \right) \quad =  \cI \left(\ell^-_t(E), \TT^1 \right).
    \]
\end{lemma}

\begin{proof}
    Since $E \notin \sigma(h_t)$  for all $t \in \TT^1$, Lemma~\ref{lem:E_eigenvalue_of_h} implies that 
    \[
        \forall t \in \TT^1, \quad \ell^+_t(E) \cap \ell^-_t(E) = \{ 0 \}.
    \]
    So the Lagrangian planes $\ell^+_t(E)$ and $\ell^-_t(E)$ never cross. In particular,
    \[
        \Mas \left( \ell^+_t(E), \ell^-_t(E), \TT^1 \right) = 0, \quad \text{and therefore} \quad
        \cI \left(\ell^+_t(E), \TT^1 \right) = \cI \left(\ell^-_t(E), \TT^1 \right).
    \]
    The proof then follows from the fact that $\cI(h_t, E) = \cI \left(\ell^+_t(E), \TT^1 \right)$ and Remark~\ref{rem:left_Maslov}.
\end{proof}

In Remark~\ref{rem:scattering}, we linked the unitary $\cU_t^+(E)$ to the reflection coefficient $R_t(k)$ with $k = \sqrt{E}$. Since they have similar Cayley transform (up to a multiplicative positive constant), the winding of $t \mapsto \cU_t^+(E)$ equals the one of $t \mapsto R_t(k)$. So our bulk/edge index is also the winding of the (determinant of the) reflection coefficient $R_t(k)$. The equality
\[
    {\rm Winding} \left( R_t(k), \TT^1 \right) = \Sf \left( h_{t, D}^\sharp, E, \TT^1 \right)
\]
can be interpreted as a weak (or integrated) form of Levinson's Theorem~\cite[Thm XI.59]{reed1980methods} (see also~\cite[Thm 6.11]{graf2013bulk}).


\subsection{Applications}
Let us give two applications of the previous theory. The first one shows that a spectral flow must appear when modifying Robin boundary conditions. The second one concerns the case of junctions between two Hill's operators.

\subsubsection{Robin boundary conditions}

\label{ssec:example_Robin}
In the case $n = 1$, consider a fixed (independent of $t$) bounded  potential $V_t(x) = V(x)$. We consider the self-adjoint Robin operators $h^\sharp_t = -\partial_{xx}^2 + V$ on $L^2(\R^+)$, with the $t$-dependent domain
\[
\cD_t := \left\{ \psi \in H^2(\R^+), \quad \sin(\pi t) \psi(0) - \cos(\pi t) \psi'(0) = 0      \right\}.
\]
We have $\cD_{t + 1} = \cD_{t}$, so $H_t^\sharp$ is $1$-periodic in $t$. For $t = 0$, we recover Dirichlet boundary conditions, and for $t = \frac12$, we recover Neumann boundary conditions, so Robin boundary conditions interpolates between these two cases. The Lagrangian plane of $\cH_b = \C \times \C$ corresponding to the extension $\cD_t$ is
\[
\ell_t^\sharp = {\rm Vect}_\C \begin{pmatrix}
\cos(\pi t) \\ \sin( \pi t)
\end{pmatrix} \subset \C \times \C.
\]
It is of the form $ \ell_t^\sharp \{ (\Theta x, \Pi x), \ x \in \C\}$ for $\Theta = \cos(\pi t)$ and $\Pi = \sin(\pi t)$. So, by the results of Example~\ref{ex:Robin}, the corresponding unitary $\cU(t) \in \U(1) \approx \SS^1$, is
\[
    \cU(t) = \dfrac{\cos(\pi t) + \ri \sin(\pi t)}{\cos(\pi t) - \ri \sin(\pi t)} = \re^{2 \ri \pi t}.
\]
We see that $\cU(t)$ winds once positively around $\SS^1$ as $t$ runs through $\TT^1$, that is
\[
    \cI(\ell_t^\sharp, \TT_1) = 1.
\]
Using Theorem~\ref{th:main-Hill}, and the fact that $\ell^+_t(E)$ is independent of $t$, we obtain
\[
    \Sf \left( h_t^\sharp, E, \TT^1  \right) = - 1.
\] 
We deduce that there is a spectral flow of exactly $1$ eigenvalue going upwards in all spectral gaps of $h$. This includes the lower gap $(-\infty, \inf \sigma(h))$.


\subsubsection{Junction between two materials.}
\label{ssec:junction}
We now consider a junction between a left and a right potentials $V_{L, t}$ and $V_{R, t}$, where $t \mapsto V_{L,t}$ and $t \mapsto V_{R, t}$ are two periodic continuously differentiable families of potentials in $L^\infty(\R, \cS_n)$. Take $\chi$ a bounded switch function, satisfying, for some $X > 0$,
\[
\forall x \le -X, \quad \chi(x) = 1, \quad \text{and} \quad \forall x \ge X, \quad \chi(x) = 0.
\]
We consider the domain wall Hill's operators
\[
h^\chi_t := - \partial_{xx}^2 + V_{L,t}(x) \chi(x) + V_{R, t}(x) (1 - \chi(x)).
\]
Let $E \in \R$ be in the resolvent set of the bulk operators $h_{R,t}$ and $h_{L,t}$ for all $t \in \TT^1$. Again, some eigenvalues of $h^\chi_t$ might cross $E$ as $t$ goes from $0$ to $1$, and we can define a corresponding spectral flow $\Sf(h_t^\chi, E, \TT^1)$. 

\begin{theorem}[Junctions between two channels] \label{th:bec_hill}
    With the previous notation, let $(a,b) \subset \R \setminus \{ \sigma \left( h_{R, t}, \TT^1 \right) \cup \sigma \left( h_{L, t}, \TT^1 \right) \}$. Then,
    \begin{itemize}
        \item almost any $E \in (a,b)$ is a regular energy for $h_t^\chi$;
        \item for any such regular energy, we have
        \begin{align*}
            \Sf(h_t^\chi, E, \TT^1) = \cI(h_{R, t}, E) -  \cI(h_{L, t}, E).
         \end{align*}
    \end{itemize}
\end{theorem}
In particular, this spectral flow is independent of the switch $\chi$.

\begin{proof}
    Let us denote by $\ell^\pm_{\chi, t}(x_0, E)$ the Lagrangian planes obtained with the potential
    \[
        V^\chi_t(x) := V_{L,t}(x) \chi(x) + V_{R, t}(x) (1 - \chi(x)),
    \]
    and when the real line $\R$ is cut at the location $x_0 \in \R$. By Lemma~\ref{lem:E_eigenvalue_of_h}, we have
    \[
       \forall x_0 \in \R, \quad  \dim \Ker \left( h_t^\chi - E \right) = \dim \left( \ell_{\chi, t}^+(x_0, E) \cap \ell_{\chi, t}^-(x_0, E) \right).
    \]  
    Adapting the proof of Theorem~\ref{th:main-Hill} shows that
    \begin{align*}
        \Sf\left(h_t^\chi, E, \TT^1 \right) & = \Mas \left( \ell^+_{\chi, t} (x_0, E), \ell^-_{\chi, t} (x_0, E), \TT^1 \right)
         \\
         & = \cI \left(   \ell^+_{\chi, t} (x_0, E), \TT^1 \right) - \cI \left(   \ell^-_{\chi, t} (x_0, E), \TT^1 \right) .
    \end{align*}
    Since $V$ is uniformly (hence locally) bounded, all Cauchy solutions to $-\psi'' + V \psi = E \psi$ are well defined and continuously differentiable on the whole line $\R$. This implies that the maps $x_0 \mapsto \ell^\pm_{\chi, t}(x_0, E)$ are also continuous. In particular, since the index depends only on the homotopy class of the loops, it is independent of $x_0 \in \R$. So
    \begin{align*}
         \cI\left( \ell^+_{\chi, t} (x_0, E), \TT^1 \right) & = \cI \left(  \ell^+_{\chi, t} (X, E), \TT^1 \right) = \cI \left( \ell_{R, t}^+(X, E), \TT^1  \right) \\
         & = \cI \left( \ell_{R, t}^+( E), \TT^1  \right) = \cI(h_{R, t}^+, E).
    \end{align*}
    For the middle equality, we used that $\ell_{\chi, t}^+(X, E)$ only involves the half space $\{ x \ge X \}$, where we have $V^\chi_t(x) = V_{R, t}(x)$. The proof for the left-hand side is similar, and the result follows from our definition of the bulk/edge index and Lemma~\ref{lem:def:bulkedge_left}.
\end{proof}


\section{The Schrödinger case}
\label{sec:Schrodinger}

In this section, we focus on the PDE Schrödinger case. We chose to put this section separately, since it introduces some technical details, and since the results are slightly different. 

\subsection{Schrödinger operators on a tube}

We consider systems defined on a $d$-dimensional cylinder of the form
\[
\Omega := \R \times \Gamma \subset \R^d,
\]
where $\Gamma = (0, 1)^{d-1}$ is the $(d-1)$-dimensional unit open square. A point in $\Omega$ is denoted by $\bx = (x, \by)$ with $x \in \R$ and $\by \in \Gamma$.

Let $V : \Omega \to \R$ be a {\em real-valued} potential, which we assume to be bounded on $\Omega$. We consider bulk Schrödinger operators  $H$ of the form
\[
H := - \Delta + V, \quad \text{acting on} \quad \cH := L^2(\Omega, \C).
\]
Again, we do not assume here that $V$ is periodic, but only that $V$ is bounded. 

The operator $H$ with core domain $C^\infty_0(\Omega)$ is symmetric, and we have
\[
\cD_{\rm min} = H^2_0(\Omega, \C), \quad \text{and} \quad
\cD_{\rm max} = H^2(\Omega, \C).
\]
This time, the bulk operator is not self-adjoint, and indeed, boundary conditions must be chosen on the boundary of the tube $\partial \Omega = \R \times \partial \Gamma$.

\subsubsection{The bulk Schrödinger operators}
\label{ssec:bulk_schrodinger}

For the sake of simplicity, we consider periodic boundary conditions. Our results hold for other boundary conditions, such as Dirichlet or Neumann, but the construction of the domains are a bit more technical. So we rather see $\Gamma$ as the torus
\[
    \Gamma := \TT^{d-1}, \quad \text{so that} \quad \Omega := \R \times \TT^{d-1}.
\]
With this definition, $\Omega$ has no boundaries: $\partial \Omega = \emptyset$, and we have $\cD_{\rm min} = \cD_{\rm max} = H^2(\Omega, \C)$. The bulk operator $H$ is now self-adjoint (corresponding to the periodic self-adjoint extension). 

For $\bk \in \ZZ^{d-1}$, we introduce the $\bk$-th Fourier mode $e_\bk(\by) := \re^{ \ri 2 \pi \bk \cdot \by}$. The elements in $\cH$ can be written in the partial Fourier form
\begin{equation} \label{eq:form_f}
    f(x, \by) = \sum_{\bk \in \ZZ^{d-1}} f_\bk(x) e_\bk(\by), \quad \text{with} \quad
    \| f \|_{\cH}^2 := \sum_{\bk \in \ZZ^{d -1}} \| f_\bk \|_{L^2(\R)}^2 < \infty.
\end{equation}
A function $f \in \cH$ is in the bulk domain $\cD := H^2(\Omega, \C)$ if $\| (-\Delta) f \|_{\cH}  < \infty$ as well, where
\[
    \| (-\Delta) f \|_{\cH}^2 = \sum_{\bk \in \ZZ^{d-1}} \left( \| f_\bk'' \|_{L^2(\R)}^2 + (4 \pi | \bk |^2)^2  \| f_\bk \|_{L^2(\R)}^2 \right) < \infty.
\]

\subsubsection{Edge Schrödinger operators on a tube}

We now define the edge Schrödinger operator
\[
H^\sharp := - \Delta + V \quad \text{acting on} \quad L^2(\Omega^+, \C), 
\quad \text{where} \quad
\Omega^+ := \R^+ \times \TT^{d-1}.
\]
This operator acts on the right half tube. We sometime write $H^{\sharp, +}$ for $H^\sharp$ and define $H^{\sharp, -}$ for the corresponding operator on the left half tube $\Omega^- := \R^- \times \TT^{d-1}$. The operator $H^\sharp$ with core domain $C^\infty_0(\Omega^+)$ is symmetric, and we have, 
\[
    \cD_{\rm min}^\sharp = H^2_0(\Omega^+, \C), \quad \text{and} \quad
    \cD_{\rm max}^\sharp = \left\{ \psi \in L^2(\Omega^+, \C), \ (- \Delta + V) \psi \in L^2(\Omega^+, \C) \right\},
\]
where the expression $(- \Delta + V) \psi$ must be understood in the distributional sense. Again, we need to specify the boundary conditions at $\partial \Omega^+ = \{ 0 \} \times \TT^{d-1}$. 

We stress out that the inclusion $H^2(\Omega^+, \C) \subset \cD_{\rm max}^\sharp$ is strict. This makes the PDE setting more tedious to describe. In this section, we focus on domains $\cD^\sharp$ which are included in $H^2(\Omega^+, \C)$ (this includes the Dirichlet and Neumann extensions). This case is well suited to study junctions, and is much simpler than the general case (with domains in $\cD_{\rm max}^\sharp$). It can be studied as for the Hill's case. We discuss the general case of domains $\cD^\sharp \subset \cD_{\rm max}^\sharp$ later in Section~\ref{sec:discussion_boundary_space}. It is based on a regularised version of Green's identity, and is well suited to study half-systems. However, the general setting is not appropriate to study junctions.

\medskip

The key ingredient in the case $\cD^\sharp \subset H^2(\Omega^+, \C)$ is the following.

\begin{lemma} \label{lem:restriction}
    A function is in $H^2(\Omega^+, \C)$ iff it is the restriction to $\Omega^+$ of an element in the bulk domain $\cD = H^2(\Omega, \C)$.
\end{lemma}

\begin{proof}
    This follows from the fact that there is an extension operator $H^2(\Omega^+) \to H^2(\Omega)$ which can be constructed with reflection operators, see {\em e.g.}~\cite[Theorem 7.25]{gilbarg2015elliptic} or~\cite[Theorem 8.1]{lions_magenes_1}. These reflection operators keep the periodic properties in the last $(d-1)$-directions.
\end{proof}

\subsection{Trace maps, and the boundary space $\cH_b$}

In order to express the second Green's identity in this setting, we recall some basic facts on the Dirichlet and Neumann trace operators. 

\subsubsection{Boundary Sobolev-like spaces}

Recall that $\Gamma = \TT^{d-1}$ is the boundary of $\Omega^+$. For $s \in \R$, we introduce the usual Hilbert spaces $H^s(\Gamma)$, with inner product 
\[
\bra f, g \ket_{H^s(\Gamma)} := \sum_{\bk \in \Z^{d-1}} \overline{f_\bk} g_\bk \left( 1 + (4 \pi | \bk |)^2 \right)^s,
\]
where we introduced the Fourier coefficients
\[
    f_\bk = \int_{\TT^{d-1}} f(\by ) \re^{ \ri 2 \pi \bk \cdot \by} \rd \by.
\]
We have $L^2(\Gamma) = H^{s=0}(\gamma)$, and for $s \ge 0$, $H^{-s}(\Gamma)$ is the dual of $H^s(\Gamma)$ for the $L^2(\Gamma)$-inner product. For $s' < s$, we have $H^s(\Gamma) \hookrightarrow H^{s'}(\Gamma)$ with compact embedding, and that $H^s(\Gamma)$ is dense in $H^{s'}(\Gamma)$. 

\subsubsection{Dirichlet and Neumann trace operators}

For $\psi \in C^\infty (\Omega^+)$, we introduce the functions $\gamma^D \psi$ and $\gamma^N \psi$ defined on $\Gamma$ by
\[
    \forall \by \in \Gamma, \quad  (\gamma^D \psi) (\by) :=  \psi( x= 0, \by), \quad \left( \gamma^N \psi \right) (\by) = \partial_x \psi(x = 0, \by).
\]
Our definition differs from the usual one $\gamma^N \psi = - \partial_x \psi(0, \cdot)$, where the minus sign comes from the outward normal direction of $\Gamma$ from the $\Omega^+$ perspective. Our definition without the minus sign matches the one of the previous section. Finally, we define the trace map
\begin{equation} \label{eq:def:Tr_S}
    \Tr(\psi)= \left( \gamma^D \psi, \gamma^N \psi  \right).
\end{equation}
It is classical that $\Tr$ can be extended as a bounded operator from $H^2(\Omega^+)$ to $H^{3/2}(\Gamma) \times H^{1/2}(\Gamma)$ (see for instance~\cite[Theorem 8.3]{lions_magenes_1}). This suggests to introduce the boundary space
\[
    \cH_b := H^{3/2}(\Gamma) \times H^{1/2}(\Gamma).
\]

The second Green's identity in the PDE case reads as follows.

\begin{lemma}[Second Green's formula] \label{lem:Green_PDE}
    For all  $\phi, \psi \in H^2(\Omega^+)$,
    \[
         \bra \phi, H^\sharp_{\rm max} \psi \ket_{\cH^\sharp} -  \bra  H^\sharp_{\rm max} \phi, \psi \ket_{\cH^\sharp}
          = \bra \gamma^D \phi, \gamma^N \psi \ket_{L^2(\Gamma)} - \bra \gamma^N \phi, \gamma^D \psi \ket_{L^2(\Gamma)}.
    \]
\end{lemma}

Introducing the symplectic form $\omega$ on $\cH_b$, defined by
\[
    \forall (f,f'), (g,g') \in \cH_b, \quad \omega( (f,f'), (g,g')) := \bra f, g' \ket_{L^2(\Gamma)} - \bra f', g \ket_{L^2(\Gamma)},
\]
the second Green's identity takes the form
\[
      \forall \phi, \psi \in H^2(\Omega^+), \quad \bra \phi, H^\sharp_{\rm max} \psi \ket_{\cH^\sharp} -  \bra  H^\sharp_{\rm max} \phi, \psi \ket_{\cH^\sharp} = \omega \left( \Tr(\phi), \Tr(\psi)  \right).
\]

\begin{remark} \label{rem:cHb'}
    The symplectic Hilbert space $(\cH_b, \omega)$ {\bf does not} satisfy Assumption~A. Introducing the map $A : H^{1/2}(\Gamma) \to H^{3/2}(\Gamma)$ so that
    \[
        \forall f \in H^{3/2}(\Gamma), \ \forall g \in H^{1/2}(\Gamma), 
        \quad
        \bra f, g \ket_{L^2(\Gamma)} = \bra f, Ag \ket_{H^{3/2}} = \bra A^* f, g \ket_{H^{1/2}},
    \]
    we have $J = \begin{pmatrix}
        0 & A^* \\ -A & 0
    \end{pmatrix}$, but the operators $A$ and $A^*$ are compact (hence $J$ as well, and $J^2 \neq - \bbI_{\cH_b}$). In particular, we cannot consider the unitaries $U$ nor $\cU$. Such situation, called {\em weak} symplectic spaces, was studied in~\cite{booss2013maslov}.
\end{remark}


\begin{lemma}
    The map $\Tr : \left(  H^2(\Omega^+), \| \cdot \|_{H^2} \right) \to \cH_b$ is well-defined, continuous and onto.
\end{lemma}

\begin{proof}
    The fact that $\Tr$ is well-defined and continuous follows from the continuity of the trace maps. To prove that $\Tr$ is onto, one can adapt the proof of~\cite[Theorem 8.3]{lions_magenes_1}. We provide here an alternative short proof.
    
    Let $f \in H^{3/2}(\Gamma)$ and $f' \in H^{1/2}(\Gamma)$ with respective coefficients $(f_\bk)$ and $(f_\bk')$. Consider also a smooth cut-off function $\chi(x)$ with $\chi(x) = 1$ for $0 \le x < 1/2$, $\chi(x) = 0$ for $ x > 2$ and $\int_{\R^+} \chi^2 = 1$. We set $\chi_\bnull = \chi$, and, for $\bk \in \Z^{d-1} \setminus \{ 0\}$,
    \[
    \chi_\bk(x) := \chi\left(  | \bk | x \right).
    \]
    For all $\bk$, the function $\chi_\bk$ is smooth, compactly supported, with $\chi_\bk(x) = 1$ for all $x < | \bk |/2$. In addition, we have the scalings
    \[
    \int_{\R^+} | \chi_\bk |^2 = \dfrac{1}{| \bk |}, \quad 
    \int_{\R^+} | \chi_\bk' |^2 =| \bk | \int_{\R+} | \chi' |^2 , \quad
    \int_{\R^+} | \chi_\bk'' |^2 = | \bk |^{3} \int_{\R^+} | \chi'' |^2.
    \]
    We now consider the function $\Psi$ defined on $\Omega^+$ by
    \[
    \Psi(x, \by) := \sum_{\bk \in \Z^{d-1}} \left( f_\bk + x f_\bk'  \right) \chi_\bk(x)e_\bk(\by).
    \]
    The function $\Psi$ is smooth with $\Tr(\Psi) = (f, f')$. It remains to check that $\Psi$ is in $H^2(\Omega^+)$. We have for instance
    \begin{align*}
    \|(- \Delta) \Psi \|_{L^2(\Omega^+)}^2 & \lesssim \int_{\R^+} \sum_{\bk \in \Z^{d-1}} \left( |f_\bk|^2 | \bk |^4 | \chi_\bk |^2 + | f_\bk |^2 | \chi_\bk'' |^2 + | f_\bk' |^2 | \chi_\bk' |^2 \right) \\
    & \lesssim \sum_{\bk \in \Z^{d-1}} | f_\bk |^2 \cdot | \bk |^{3} + \sum_{\bk \in \Z^{d-1}} | f_\bk' |^2 | \bk | \lesssim \| f \|_{H^{3/2}}^2 + \| f' \|_{H^{1/2}}^2.
    \end{align*}
    where we used our previous scalings for $\chi_\bk$. The $L^2$--norms of the other derivatives are controlled similarly.
\end{proof}

\subsubsection{Self-adjointness and Lagrangian spaces}
We now provide the counterparts of our previous results in the Schrödinger case. First, we have (compare with Theorem~\ref{th:self_adjoint_extensions_are_Lagrangian_planes})

\begin{theorem}\label{th:self_adjoint_extensions_are_Lagrangian_planes_Schrodinger}
    If $\cD^\sharp \subset H^2(\Omega^+)$ is a domain so that $(H^\sharp_{\rm max}, \cD^\sharp)$ is self-adjoint, then $\ell := \Tr(\cD^\sharp)$ is a Lagrangian plane of $\cH_b$.
\end{theorem}

\begin{proof}
    The proof follows the one of Theorem~\ref{th:self_adjoint_extensions_are_Lagrangian_planes}. First, the second Green's identity shows that
    \[
        \forall \psi, \phi \in \cD^\sharp, \quad 0 = \bra \psi, H^\sharp \phi \ket - \bra H^\sharp \psi , \phi \ket = \omega \left( \Tr(\psi), \Tr(\phi) \right),
    \]
    hence $\ell \subset \ell^\circ$. Conversely, if $\psi_0 \in H^2(\Omega^+)$ is such that $\Tr(\psi_0) \in \ell^\circ$, then for all $\phi \in \cD^\sharp$, we have
    \[
        0 = \omega \left( \Tr(\psi_0), \Tr(\phi) \right) = \bra \psi_0, H^\sharp \phi \ket - \bra H^\sharp \psi_0 , \phi \ket.
    \]
    In particular, the map $T_{\psi_0} : \phi \mapsto \bra \psi_0, H^\sharp \phi \ket = \bra H^\sharp \psi_0, \phi \ket$ is bounded, with $\| T_{\psi_0} \|_{\rm op} \le \| H^\sharp \psi_0 \|_{\cH^\sharp}$. So $\psi_0 \in (\cD^\sharp)^* = \cD^\sharp$. This proves that $\ell^\circ \subset \Tr(\cD^\sharp) \subset \ell$, hence $\ell = \ell^\circ$ is Lagrangian.
\end{proof}

Theorem~\ref{th:self_adjoint_extensions_are_Lagrangian_planes_Schrodinger} is a much weaker statement than Theorem~\ref{th:self_adjoint_extensions_are_Lagrangian_planes}, but is still enough for our purpose (in practice, the self-adjoint extensions are given). There is no longer a one-to-one correspondence between Lagrangian planes and self-adjoint extensions. One problem is that, for $\ell \subset \cH_b$, although $\Tr^{-1}(\ell)$ is included in $H^2(\Omega^+)$, its closure for the graph norm $\overline{\Tr^{-1}(\ell)}$ may no longer be included in $H^2(\Omega^+)$. We refer to~\cite[Example 4.22]{Behrndt_2014} for an example of such a situation.

\medskip

The problem of recovering a function $\psi \in \cD^\sharp$ from its boundary value $\Tr(\psi)$ is a well-known problem, often called ``boundary value problem'', which has been extensively studied in the literature. The modern tool for this problem is the notion of {\em boundary triples}~\cite{Behrndt_2007, Behrndt_2020}. In the terminology of the community, we have, in the Hill's case, that $(\cH_b := \C^{2n}, \Tr^D, \Tr^N)$ is an ordinary boundary triple, while in the Schrödinger case, $(\cH_b = H^{3/2}(\partial \Omega) \times H^{1/2}(\partial \Omega), \Tr^D, \Tr^N)$ is a {\em quasi}-boundary triple~\cite{Behrndt_2014}. Below in Section~\ref{sec:discussion_boundary_space}, we prove a one-to-one correspondence between all self-adjoint extensions and Lagrangian planes of another symplectic Hilbert space of the form $\widetilde{\cH_b} = H^{-1/2}(\Gamma) \times H^{1/2}(\Gamma)$. Unfortunately, this construction uses a regularization of the Neumann trace, introduced by Vishik~\cite{vishik1952general} and Grubb~\cite{grubb1968characterization}, and it is not well suited for the study of junctions, as discussed in Section~\ref{sec:discussion_boundary_space}.

\medskip

If $\cD^\sharp \subset H^2(\Omega^+)$ is a self-adjoint domain, and if $\ell^\sharp = \Tr(\cD^\sharp)$ is the corresponding Lagrangian plane, we denote by $(H^\sharp, \ell^\sharp)$ the self-adjoint extensions of $H^\sharp$ corresponding to this Lagrangian plane. Not all $\ell^\sharp \subset \Lambda(\cH_b)$ define a self-adjoint domain. As in the finite dimensional case, we define
\begin{equation} \label{eq:def:SE}
\cS^\pm(E) := {\rm Ker} \left( H^{\sharp, \pm}_{\rm max} - E \right) \cap H^2(\Omega^\pm), \quad \text{and} \quad
\ell^\pm(E) := \Tr \left( \cS^\pm(E) \right).
\end{equation}

The counterpart of Lemma~\ref{lem:E_eigenvalue_of_h} is the following.
\begin{lemma} \label{lem:E_eigenvalue_of_H}
    For the bulk operator $H$, we have 
    \[
    \forall E \in \R, \quad \dim \Ker \left(H - E\right) = \dim \left( \ell^+(E) \cap \ell^-(E) \right).
    \]
    In particular, $E$ is eigenvalue of $H$ iff $\ell^+(E) \cap \ell^-(E) \neq \{ 0 \}$.
\end{lemma}


\begin{proof}
    If $\psi \in \cD$ satisfies $(H - E) \psi = 0$, then, by Lemma~\ref{lem:restriction}, its restrictions $\psi^\pm := \1_{\R^\pm} \psi$ are in $H^2(\Omega^\pm)$. In addition, they satisfy $(H_{\rm max}^{\sharp} - E) \psi^\pm = 0$, so  $\psi^\pm \in \cS^\pm(E)$. Taking traces shows that $\Tr^+(\psi^+) = \Tr^-(\psi^-) \in \ell^+(E) \cap \ell^-(E)$.
    
    Conversely, let $\psi^\pm \in \cS^\pm(E)$ be such that $\Tr^+(\psi^+) = \Tr^-(\psi^-)$, and consider the function $\psi \in \cH$ defined by
    \[
    \psi(x,y) := \begin{cases}
        \psi^+(x,y) \quad \text{for} \quad x > 0, \\
        \psi^-(x,y) \quad \text{for} \quad x < 0.
    \end{cases}
    \]
    It is unclear yet that $\psi$ is regular enough ({\em i.e.} belongs to $\cD = H^2(\Omega)$). For $f \in \cD$, we have
    \begin{align*}
        \bra \psi, (H -E )  f \ket_{\cH}
        & = \bra \psi^+,  \1_{\R^+} (H -E )  f \ket_{\cH^+} + \bra \psi^-,  \1_{\R^-} (H -E )  f \ket_{\cH^-} \\
        & = \bra \psi^+,  (H_{\rm max}^{\sharp, +} - E)  f^+ \ket_{\cH^+} + \bra \psi^-,  (H_{\rm max}^{\sharp, -} - E)  f^- \ket_{\cH^-} \\
        & = \omega \left( \Tr^+(\psi^+), \Tr^+(f^+) \right) - \omega \left( \Tr^-(\psi^-), \Tr^-(f^-) \right) = 0.
    \end{align*}
    So $T_\psi : f \mapsto \bra \psi, H f \ket_{\cH} = E \bra \psi, f \ket_{\cH}$ is bounded on $\cD$. We first deduce that $\psi$ is in the domain $\cD^* = \cD$. In addition, we have $(H - E) \psi = 0$. So $\psi$ is an eigenvector for the eigenvalue~$E$.
\end{proof}

\begin{theorem} \label{th:Lagrangian_planes_Schrodinger_lE}
    For all $E \in \R \setminus \sigma(H)$, the sets $\ell^\pm(E)$ are Lagrangian planes of $\cH_b$, and
    \[
    \cH_b = \ell^+(E) \oplus \ell^-(E).
    \]    
\end{theorem} 

If $E \in \sigma(H)$, the planes $\ell^\pm(E)$ may not be Lagrangian (see Remark~\ref{rem:not-always-lagrangian-planes}).

\begin{proof}
    We first claim that for any $E \in \R$, $\ell^\pm(E)$ are isotropic spaces. Let $\phi, \psi \in \cS^+(E)$. By Green's identity, we have
    \[
    \omega( \Tr(\phi), \Tr(\psi)) = \bra \phi, H_{\rm max}^\sharp \psi \ket_{\cH^\sharp} - \bra H_{\rm max}^\sharp \phi, \psi \ket_{\cH^\sharp}
    = \bra \phi, E \psi \ket_{\cH^\sharp} - \bra E \phi, \psi \ket_{\cH^\sharp} = 0.
    \]
    In the last equality, we used that $E$ is real-valued. This proves that $\ell^+(E) \subset \ell^+(E)^\circ$. Similarly, we have $\ell^-(E) \subset \ell^-(E)^\circ$.
    
    \medskip
    
    We have (recall that $\cH = L^2(\Omega, \C^n)$)
    \[
    \cH = \cH^+ \oplus \cH^-, \quad \text{where} \quad \cH^\pm := \left\{ \psi \in \cH, \ \psi = 0 \ \text{on} \ \overline{\Omega^\mp} \right\}.
    \]
    Let $E \in \R \setminus \sigma(H)$, so that the bulk operator $(H - E)$ is invertible with $\cD = (H - E)^{-1} \cH$. This gives a decomposition
    \[
    \cD = \cD^+ \oplus \cD^-, \quad \cD^\pm := (H - E)^{-1} \cH^\pm,
    \]
    and, since $\Tr$ is onto,
    \[
    \cH_b = \Tr \left( \cD^+ \right) + \Tr \left(\cD^- \right).
    \]
    The elements $\psi \in \cD^+$ are such that $(- \Delta + V - E) \psi = f$, for some $f \in \cH$ with support contained in $\Omega^+$. In particular, the restriction of $\psi$ to $\Omega^-$, denoted by $\psi^-$, is in $H^2(\Omega^-)$ and satisfies $(H_{\rm max}^{\sharp, -} - E) \psi^- = 0$ on $\R^-$. So $\psi^- \in \cS^-(E)$. Taking boundary traces shows that
    \[
    \Tr \left( \cD^+ \right)  \subset \ell^-(E), 
    \quad \text{and, similarly, } \quad 
    \Tr \left( \cD^- \right)  \subset \ell^+(E).
    \]
    In particular, we have $\cH_b = \ell^+(E) + \ell^-(E)$. We conclude with Lemma~\ref{lem:useful_result_lagrangians}.
    
\end{proof}

Finally, the counterpart of Lemma~\ref{lem:dimKer=dimCap} is the following. We skip the proof for the sake of brevity. 
\begin{lemma}
    If $E \in \R \setminus \sigma ( H )$, then, 
    \[
    \dim \Ker \left(  H^\sharp - E \right) = \dim \left( \ell^+(E) \cap \ell^\sharp  \right).
    \]
\end{lemma}

In the finite dimensional Hill's case, for all extensions $(h^{\sharp, +}, \ell^\sharp)$ and $(h^{\sharp, -}, \ell^{\sharp})$ with the same Lagrangian plane $\ell^\sharp$, we have 
\[
    \sigma_{\rm ess}(h) = \sigma_{\rm ess}(h^{\sharp,+}) \cup \sigma_{\rm ess}(h^{\sharp,-}).
\] 
This is because boundary conditions always induce finite dimensional (hence compact) perturbations of the resolvents. In the Schrödinger case, we only have the inclusion
\[
    \sigma_{\rm ess} (H) \subset \sigma_{\rm ess}(H^{\sharp,+}) \cup  \sigma_{\rm ess}(H^{\sharp,-}),
\]
which comes from the fact that Weyl sequences for $H$ must escape to $\pm \infty$. However, the inclusion may be strict: in the infinite dimensional case, there are self-adjoint extensions of $H^\sharp$ which can create essential spectrum. The corresponding Weyl sequences localise near the cut. We give such an example below in Remark~\ref{rem:creation_essential_spectrum}. 

This makes bulk-boundary correspondence more subtle in the Schrödinger case: different self-adjoint extensions may give different results. For the usual extensions however, we prove that the result are independent of the choice (see the proof of Theorem~\ref{th:bec_schrodinger} below).


\subsubsection{Families of Schrödinger operators}
\label{ssec:families_Schrodinger}

We consider a family of Schrödinger operators of the form
\[
H_t := - \Delta + V_t, \quad \text{acting on} \quad \cH.
\]
We assume that $t \mapsto V_t$ is continuously differentiable from $\TT^1$ to $L^\infty(\Omega, \R)$. We also consider a family of (self-adjoint extensions of) edge operators $\left( H_t^\sharp, \ell_t^\sharp \right)$.

Let $E \in \R \setminus \sigma \left( H_t  \right)$. We say that $E$ is a {\bf regular energy} if, for all $t \in \TT^1$, the energy $E$ is not in the essential spectrum of $H^\sharp_t$. In particular, this implies $\dim \Ker \left( H^\sharp_t - E  \right)  = \dim \left( \ell_t^+(E) \cap \ell^\sharp_t \right) < \infty$. In addition, we require all corresponding crossings to be regular.

Noticing that the definition of the Maslov index in Section~\ref{ssec:Maslov} does not require Assumption A, we can apply the first part of the proof of Theorem~\ref{th:main-Hill} to the Schrödinger case, and we obtain the following.

\begin{theorem} \label{th:main_schrodinger}
    Let $(a, b) \subset \R$ be such that, for all $t \in \TT^1$,
    \[
    (a, b) \cap \sigma \left( H_t \right) = \emptyset \quad \text{and} \quad (a, b) \cap \sigma_{\rm ess} \left( H_t^\sharp \right) = \emptyset.
    \]
    Then, 
    \begin{itemize}
        \item almost any $E \in (a,b)$ is a regular energy;
        \item for such a regular energy, we have
        \begin{align*}
        \Sf \left( H_t^\sharp, E, \TT^1  \right) 
        & = \Mas \left( \ell^+_t(E), \ell^\sharp_t, \TT^1  \right).
        \end{align*}
    \end{itemize}
\end{theorem}

The proof is similar to the one of Theorem~\ref{th:main-Hill}, by noticing that all crossings involve finite dimensional linear spaces. Since the symplectic space $(\cH_b, \omega)$ does not satisfy Assumption A, it is unclear whether one can interpret this last index as a spectral flow of unitaries. We postpone this question to Section~\ref{sec:discussion_boundary_space} below.

\subsection{Junctions for Schrödinger operators}
In Section~\ref{ssec:junction}, we proved that the spectral flow for the junctions of two Hill's operators is the difference between a right and a left contributions (the index splits). We prove a similar result for Schrödinger operators.

\subsubsection{Bulk/edge index} First we define our bulk/edge index. As in Definition~\ref{def:bulk_edge}, we define it as the spectral flow for the corresponding Dirichlet edge operator.
\begin{definition}[Bulk/edge index -bis]
    We define the {\em bulk/edge index} of the family of {\bf bulk} operators $(H_t)_{t \in \TT^1}$ at energy $E \notin \sigma (H_t)$ as the spectral flow of its (right) Dirichlet {\bf edge} restriction:
    \[
        \boxed{ \cI \left( H_t, E \right) := \Sf \left( H^{\sharp, +}_{t, D}, E, \TT^1 \right).}
    \]
\end{definition}

Let us prove that this definition indeed makes sense, and in particular that Dirichlet boundary conditions does not create essential spectrum. We set $\ell_D^\sharp := \{ 0 \} \times H^{1/2} \in \Lambda (\cH_b)$ the Lagrangian plane corresponding to Dirichlet boundary conditions (that is with domain $H^2(\Omega^+) \cap H^1_0(\Omega^+)$).

\begin{theorem}\label{th:bulkedge_well_defined}
    For all $E \notin \sigma(H_t)$, the spectral flow $\Sf  \left( H^{\sharp, +}_{t, D}, E, \TT^1 \right)$ is well defined. In addition, we have
    \begin{align*}
        \cI \left( H_t, E  \right) = \Sf \left( H^{\sharp, +}_{t, D}, E, \TT^1 \right) 
         = \Mas \left( \ell_{t}^\pm(E) , \ell^\sharp_D, \TT^1 \right) 
        = -  \Sf \left( H^{\sharp, -}_{t, D}, E, \TT^1 \right),
    \end{align*}
\end{theorem}

\begin{proof}
    Let $H_{t,D} := - \Delta + V_t$ be the operator acting on $L^2 (\Omega) \approx L^2(\Omega^- \cup \Omega^+)$, but with Dirichlet boundary conditions at $\{0 \} \times \Gamma$. Since $V_t$ is uniformly  bounded, the operators $H_t$ and $H_{t,D}$ are uniformly bounded from below. Consider $\Sigma \in \R$ such that
    \[
    \Sigma < \inf_{t \in \TT^1} \inf \sigma( H_t)  \quad \text{and} \quad \Sigma < \inf_{t \in \TT^1} \inf \sigma( H_{D,t}).
    \]
    We set $R_t := (H_t - \Sigma)^{-1}$ and $R_{t,D} := (H_{t,D} - \Sigma)^{-1}$, which are both bounded operators. It is a standard result (see for instance~\cite[Theorem XI.79]{reed1980methods} or~\cite{carron2002determinant}) that, for some $m \in \N$, $R_t^m - R_{t,D}^m$ is a compact (even trace-class) operator. In particular, for all $t \in \TT^1$, we have
    \[
    \sigma_{\rm ess}(H_t) = \sigma_{\rm ess}(H_{t,D}).
    \]
    Let $(a,b)$ denote an essential gap of these operators, and let $E \in (a,b)$ be a regular energy for both operators. We see that a branch of eigenvalues of $H_t$ crosses the energy $E$ downwards iff a branch of eigenvalues of $(H_t - \Sigma)^{-m}$ crosses $(E - \Sigma)^{-m}$ upwards. So we have
    \[
        \Sf \left( H_t, E, \TT^1  \right) = - \Sf \left( R_t^m, (\Sigma - E)^{-m}, \TT^1 \right)
    \]
    and similarly for $H_{t,D}$. Since $E \notin \sigma(H_t)$, we have $\Sf \left( H_t, E, \TT^1  \right)  = 0$. Introducing
    \[
    R^m_t(s) := R_t + s(R_{t, D}^m - R_t^m),
    \]
    we see that $s\mapsto R^m_t(s)$ is a continuous family of operators connecting $R_t^m$ and $R_{t, D}^m$. Since for all $s \in [0, 1]$, $R^m_t(s)$ is a compact perturbation of $R_t$, the essential gap does not close as $s$ varies. We deduce that the spectral flow of $t \mapsto R_t^m(s)$ is independent of $s$ (see for instance~\cite[Proposition 3]{phillips1996self} or~\cite[Lemma 4]{gontier2021spectral}). So
    \[
    \Sf \left( R_t^m, (\Sigma - E)^{-m}, \TT^1 \right) = \Sf \left( R_{t, D}^m, (\Sigma - E)^{-m}, \TT^1 \right),
    \]
    which gives
    \[
        0 = \Sf \left( H_t, E, \TT^1  \right) = \Sf \left( H_{t, D}, E, \TT^1  \right).
    \]
     The operator $H_{t, D}$ decouples the left and the right side, so $E$ is an eigenvalue of $H_{t, D}$ iff it is an eigenvalue of either $H^{\sharp, +}_{t, D}$ or $H^{\sharp, -}_{t, D}$. Actually, we have
    \begin{equation*}
        \Sf \left( H_{D, t}, E, \TT^1  \right) = \Sf \left( H_{D, t}^{\sharp, +}, E, \TT^1  \right) + \Sf \left( H_{D, t}^{\sharp, -}, E, \TT^1  \right),
    \end{equation*}
    and the result follows.
\end{proof}

\subsubsection{Junction case} Let us consider two families of potentials $V_{L, t}$ and $V_{R, t}$, continuously differentiable from $\TT^1$ to $L^\infty(\Omega)$. For $\chi : \Omega \to [0, 1]$ a bounded switch function with $\chi(x, \by) = 1$ for $x < -X$ and $\chi(x, \by) = 0$ for $x > X$, we set
\[
    H_t^\chi := - \Delta + V^\chi_t(\bx), \quad \text{with} \quad V^\chi_t := V_{L, t} \chi + V_{R, t} (1 - \chi).
\]
As in Section~\ref{ssec:junction}, $H_t^\chi$ models a junction between a left and right potential.

We denote by $H_{L, t}$ and $H_{R, t}$ the corresponding left and right Hamiltonians.  Let $E \in \R$ be in the resolvent set of both $H_{L, t}$ and $H_{R, t}$ for all $t \in \TT^1$, so that the Lagrangian planes $\ell^\pm_{L, t}(E)$ and $\ell^\pm_{R, t}(E)$ are well-defined. 

\begin{theorem}[Junctions in the Schrödinger case] \label{th:bec_schrodinger}
    Let $(a,b) \subset \R$ be such that, for all $t \in \TT^1$,
    \[
        (a,b) \cap \left( \sigma(H_{t, R}) \cup \sigma(H_{t, L}) \right)  = \emptyset.
    \]
    Then, for all $E \in (a,b)$, we have
    \[
        \Sf \left( H_t^\chi, E, \TT^1 \right) = \Mas \left( \ell_R^+(E), \ell_L^-(E), \TT^1  \right)
         = \cI \left( H_{R, t}, E \right) - \cI \left( H_{L, t}, E \right) .
    \]
    This number is independent of $\chi$ and of $E$ in the gap.
\end{theorem}

\begin{proof}
    We first prove the result for $\chi_0(x) = \1(x < 0)$. Reasoning as in the proof of Theorem~\ref{th:bulkedge_well_defined}, we obtain
    \begin{equation*}
          \Sf \left( H_{t}^{\chi_0}, E, \TT^1  \right) 
          = \Sf \left( H_{t, D}^{\chi_0}, E, \TT^1  \right)
           = \Sf \left( H_{t, D}^{\chi_0, \sharp, +}, E, \TT^1  \right) + \Sf \left( H_{t, D}^{\chi_0, \sharp, -}, E, \TT^1  \right).
    \end{equation*}
    Noticing that $H_{t, D}^{\chi_0, \sharp, +}$ only depends on the right part of the potential, while $H_{t, D}^{\chi_0, \sharp, -}$ depends on the left part, together with our definition of the bulk/edge index, we get
    \begin{equation*}
         \Sf \left( H_{t}^{\chi_0}, E, \TT^1  \right) = \cI \left( H_{R, t}, E \right) -  \cI \left( H_{L, t}, E \right).
    \end{equation*}
  
  For a general switch function $\chi$, the function $\chi - \chi_0$ is compactly supported. In particular,
  \[
    H_t^\chi - H_t^{\chi_0} = (V_{L, t} - V_{R, t}) (\chi - \chi_0),
  \]
  is a compact perturbation of $H_t^{\chi_0}$ for all $t \in \TT_1$. Again, by robustness of the spectral flow with respect to compact perturbation, we obtain that $\Sf(H_t^\chi, E, \TT^1) = \Sf(H_t^{\chi_0}, E, \TT^1)$, which is independent of the switch $\chi$.
\end{proof}

\begin{remark}[Neumann boundary condition]
We defined the bulk/edge index $\cI(H_t, E)$ as the spectral flow of the {\em Dirichlet} boundary conditions (denoted by $\sharp_D$ here)
\[  
    \cI(H_t, E) = \Sf \left( H_{t}^{\sharp_D , +}, E, \TT^1\right).
\]
One can wonder what happens if one takes another (fixed) boundary condition, say Neumann (denoted by $\sharp_N$). By~\cite[Theorem XI.80]{reed1980methods}, $R_t^m - R_{N, t}^m$ is also compact for some $m > 0$, where $R_{N,t} := (H_{N,t} - \Sigma)^{-1}$ is defined as $R_{D, t}$, but with Neumann boundary conditions. Following the proof of Theorem~\ref{th:bec_schrodinger}, we obtain
\begin{equation*}
    \Sf \left( H_{t}^\chi, E, \TT^1  \right) = \Sf \left( H_{R, t}^{\sharp_N , +}, E, \TT^1\right) +  \Sf \left( H_{L, t}^{\sharp_N , -}, E, \TT^1\right),
\end{equation*}
Taking $H_{L,t}$ independent of $t$, dropping the notation $R$ for $H_{R, t}$, and using Theorem~\ref{th:bec_schrodinger} shows that
\[
\Sf \left( H_{t}^{\sharp_N , +}, E, \TT^1\right) = \Sf \left( H_{ t}^{\sharp_D , +}, E, \TT^1\right).
\]
In other words, the bulk/edge index defined with Neumann boundary conditions equals the one with Dirichlet boundary conditions. \\
This reasoning can be generalised for other boundary conditions, but not all of them (as was the case in the finite dimensional case), since some extensions might create essential spectrum, as we already mentioned.
\end{remark}


\subsection{Two-dimensional materials}
We now explain how to extend our results for the important case of two-dimensional materials. We write $\bx = (x,y) \in \R^2$. Let $V : \R^2 \to \R$ be a $\Z^2$-periodic bounded potential, and let $b : \R^2 \to \R$ be a bounded $\Z^2$-periodic magnetic field. A two-dimensional material with potential $V$ and under the magnetic field $b$ perpendicular to the plane is usually modelled by a Schrödinger operator of the form
\[
    \widetilde{H} := (- \ri \nabla + \bA(\bx))^2 + V(\bx) \quad \text{acting on $L^2(\R^2)$},
\]
where the magnetic vector potential $\bA = (A_1, A_2)^T$ satisfies $\partial_{x} A_2 - \partial_{y} A_1 = b(\bx)$. In the case where $b$ only depends on the $x$-direction $b(\bx) = b(x)$ (this is the case for constant magnetic fields for instance), we can choose the gauge
\[
    \bA = (0, A(x))^T, \quad \text{with} \quad A(x) := \int_0^{x} b(t) \rd t.
\]
The operator $\widetilde{H}$ then takes the form
\[
    \widetilde{H} := - \partial_{xx}^2 + (- \ri \partial_y + A(x))^2 + V(\bx).
\]
Since the operator $\widetilde{H}$ commutes with $\Z$-translations in the $y$-direction, one can perform a partial Bloch transform~\cite{reed1978analysis} in this direction. One obtains the operators 
\[
    \widetilde{H}_t := - \partial_{xx}^2 + (- \ri \partial_y + A(x) + 2 \pi t)^2 + V(\bx),
\]
where $k := 2 \pi t$ corresponds to the Bloch quasi-momentum in the $y$-direction. The operators $\widetilde{H}_t$ are essentially self-adjoint on $L^2(\R \times \TT^1, \C)$, with domain $H^2(\R \times \TT^1, \C)$. When these operators are cut, we get the operators
\[
    \widetilde{H}^\sharp_t :=  - \partial_{xx}^2 + (- \ri \partial_y + A(x) + 2 \pi t)^2 + V(\bx), 
    \quad \text{acting on} \quad L^2 \left( \R^+ \times \TT^1, \C   \right).
\]
These operators are not essentially self-adjoint, and the minimal/maximal domains are respectively given by
\[
    \widetilde{\cD}^\sharp_{\rm min} = \cD^\sharp_{\rm min} = H^2_0(\Omega^+) \quad \text{and} \quad
    \widetilde{\cD}^\sharp_{\rm max} = \cD^\sharp_{\rm max} = H^2(\Omega^+),
\]
independent of $t$. Although the kinetic operator now depends on $t$, it only twists the functions in the direction parallel to the cut. In particular, the second Green's identity in Lemma~\ref{lem:Green_PDE} still holds: 
\[
   \forall \phi, \psi \in \widetilde{\cD}^\sharp_{\rm max}, \quad  \bra \phi, \widetilde{H}^\sharp_t \psi \ket_{L^2(\Omega^+)} - \bra \widetilde{H}^\sharp_t \phi, \psi \ket_{L^2(\Omega^+)} = \omega(\Tr \, \phi, \Tr \, \psi),
\]
with the same $\Tr$ map and the same $\omega$ symplectic form as in the previous section (independent of $t$).

In particular, all previous results stated for the operators $H_t^\sharp$ also hold for the operators $\widetilde{H}_t^\sharp$. There is a slight abuse of notation concerning the spectral flow: the family $t \mapsto \widetilde{H}_t$ is not $1$ periodic but {\em quasi-periodic}, in the sense that $\widetilde{H}_{t+1}$ is unitary equivalent to $\widetilde{H}_t$:
\[
    \widetilde{H}_{t+1} = S^* \widetilde{H}_t S, \quad \text{with the unitary $S$ defined by} \quad (Sf)(x,y) := \re^{ 2 \ri \pi y} f(x,y).
\] 
A similar relation holds for the Dirichlet edge operator $\widetilde{H}_{t, D}^\sharp$, since the Dirichlet domain is invariant by the $S$ operator. The spectra $\sigma(\widetilde{H}_t)$ and $\sigma(\widetilde{H}_{t, D}^\sharp)$ are still $1$-periodic, and we can again define the spectral flow of such quasi-periodic family of operators as the number of eigenvalues going downwards in a gap. This allows to define the bulk/edge index for the operators $\widetilde{H}_t$.

Let us consider a junction between two such materials, of the form
\[
\widetilde{H}_t^{\rm junct} :=  - \partial_{xx}^2 + (- \ri \partial_y + A^{\rm junct}(x) + 2 \pi t)^2 + V^{\rm junct}(\bx).
\]
where $A^{\rm junct}(x)$ and $V^{\rm junct}(\bx)$ are so that
\[
    \left( A^{\rm junct}, V^{\rm junc} \right) = \begin{cases}
        (A^L, V^L), & x < -X, \\
        (A^R, V^R), & x > X,
    \end{cases}
\]
for some $X > 0$. Defining $\widetilde{H}_t^{L/R} :=  - \partial_{xx}^2 + (- \ri \partial_y + A^{L/R}(x) + 2 \pi t)^2 + V^{L/R}(\bx)$, we can prove as before that, for all $E \notin \left( \sigma (\widetilde{H}_t^L) \cap \sigma (\widetilde{H}_t^R) \right)$, we have
\[
    \boxed{ \Sf \left( \widetilde{H}_t^{\rm junct}, E, \TT^1 \right) =  \cI \left( \widetilde{H}_t^{R}, E \right) -  \cI \left( \widetilde{H}_t^{L}, E \right). }
\]
We do not repeat the proof, as it is similar to the one of Theorem~\ref{th:bec_schrodinger}.

\begin{example}[Landau Hamiltonian]
    Assume $V = 0$, and $b \in \R^*$ is constant. We are studying the Landau Hamiltonian 
    \[
        \widetilde{H} := - \partial_{xx}^2 + ( - \ri \partial_y + b x )^2, \quad \text{acting on } \quad L^2(\R^2).
    \]
    It is well-known that $H$ has a discrete (essential) spectrum $\sigma(H)  = | b | (2 \N_0 + 1)$. Applying a Bloch transform (instead of the usual Fourier transform) in the $y$-direction gives the operators
    \[
    \widetilde{H}_t := - \partial_{xx}^2 + (- \ri \partial_y + b x + 2 \pi t)^2 = - \partial_{xx}^2 + \left(- \ri \partial_y + b (x + \tfrac{2 \pi t}{b}) \right)^2,
    \]
    which are all unitarily equivalent to $\widetilde{H}_{t=0}$, up to the translation $x \mapsto x + \tfrac{2 \pi t}{b}$. We recognise a {\em charge pumping} phenomenon~\cite{thouless1983quantization}, where the system undergoes a translation in the $x$--direction of $-\tfrac{2 \pi}{b}$ as $t$ goes from $0$ to~$1$.
    Let $E \in \R \setminus \sigma(H)$, and let $\cN_b(E)$ be the number of Landau bands below $E$, that is $\cN_b(E) = \lceil \frac12 \left( \tfrac{E}{| b |} - 1 \right) \rceil$. Each Landau band has a constant electronic density $\frac{| b |}{2 \pi}$, in the sense that there are $\frac{| b |}{2 \pi }$ electrons per unit cell in each Landau band. So, as $t$ goes from $0$ to $1$, the total charge which is pumped below $E$ is $\frac{| b |}{2 \pi} \cN_b(E) \times (-\tfrac{2 \pi}{b}) = - {\rm sign}(b) \cN_b(E)$. Reasoning as in~\cite{hempel2011variational, hempel2012dislocation, gontier2021spectral}, we deduce that
    \[
    \cI \left( \widetilde{H_t}, E  \right) = -  {\rm sign}(b) \cN_b(E) = -  {\rm sign}(b) \left\lceil \frac12 \left( \frac{E}{|b|} - 1 \right) \right\rceil.
    \]
    Let $\widetilde{H}^\sharp$ be the Landau operator on the half space $L^2(\R^+ \times \R)$ with Dirichlet (or Neumann) boundary condition at $x = 0$. The previous result shows that, for the family $\widetilde{H}^\sharp_t$, there is a spectral flow of $\cN_b(E)$ eigenvalues going upwards (if $b > 0$) or downwards (if $b < 0$) in the gap containing $E$, as $t$ goes from $0$ to $1$. In particular, all ``bulk'' gaps of $\widetilde{H}^\sharp$ are filled with ``edge'' spectrum, a well-known result~\cite{De_Bi_vre_2002}. \\
    
    Let $V(x,y)$ be a bounded external potential, $\Z$-periodic in the $y$ variable. For $s > 0$, we denote by
    \[
        \widetilde{H}(s) := - \partial_{xx}^2 + \left(- \ri \partial_y + b x \right)^2 + s V(x,y).
    \]
    This operator still commutes with $\Z$-translations in the $y$-variable, so we can apply a partial Bloch transform. For $s_0 > 0$ small enough, $E$ is in the resolvent set of $H(s)$ for all $s \in [0, s_0]$. So $\cI(\widetilde{H_t}(s), E) $ is independent of $s$ (see proof of Theorem~\ref{th:bulkedge_well_defined} above), and $\cI(\widetilde{H_t}(s), E) = - {\rm sign}(b) \cN_b(E)$ as well, for all $s \in [0, s_0]$.
\end{example}

It would be interesting to relate our bulk/edge index $\cI \left( \widetilde{H}_t, E \right)$ to a bulk index of the operator $\widetilde{H}$ (for instance to a Chern number or Chern marker), in the general case.

\subsection{General self-adjoint extensions}
\label{sec:discussion_boundary_space}

In this section, we introduce another symplectic boundary space $(\widetilde{\cH_b}, \widetilde{\omega})$ and another trace map $\widetilde{\Tr}$, which allows to treat the general case of self-adjoint extensions of $H^\sharp_{\rm min}$ with domains $\cD^\sharp \subset \cD_{\rm max}^\sharp$ (not necessarily included in $H^2(\Omega^+)$). 

The main idea of the section is to use a Green's identity involving a regularized Neumann trace. This was first introduced by Vishik~\cite{vishik1952general} and Grubb~\cite{grubb1968characterization}. We skip most of the proofs of this section, and refer to the monograph~\cite{Behrndt_2020} for details. Similar ideas have been used in the context of elliptic operators in~\cite{deng2011multi} (see also~\cite{deng2006infinite, deng2008infinite}).

\subsubsection{The regularized Green's formula}
Recall that
\[
    \cD_{\rm max}^\sharp = \left\{ \psi \in L^2(\Omega^+), \quad (- \Delta + V) \psi \in L^2(\Omega^+ )\right\}.
\]
For any $E \in \R$, we introduce the null space
\[
       \widetilde{\cS_E} := {\rm Ker}(\cD_{\rm max}^\sharp - E) = \left\{ \psi \in \cD_{\rm max}^\sharp, \ ( - \Delta + V) \psi = E \psi  \right\}.
\]
The space $\cS_E$ introduced in~\eqref{eq:def:SE} is $\cS_E = \widetilde{\cS_E} \cap H^2(\Omega^+)$.

\medskip

Let $H_D^\sharp$ be the Dirichlet extension of $( - \Delta + V)$ on $L^2(\Omega^+)$, that is with domain $\cD_D^\sharp := H^2(\Omega^+) \cap H^1_0(\Omega^+)$, and let $\Sigma \in \R \setminus \sigma (H_D^\sharp)$ be a fixed energy in the resolvent set of $H_D^\sharp$. For $\psi \in \cD_{\rm max}^\sharp$, we set
\[
    \psi_D := (H_D^\sharp - \Sigma)^{-1} (H_{\rm max}^\sharp - \Sigma) \psi \quad \in H^2(\Omega^+) \cap H^1_0(\Omega^+),
\]
and
\[
    \psi_\Sigma := \psi - \psi_D = \left( \bbI -   (H_D^\sharp - \Sigma)^{-1} (H_{\rm max}^\sharp - \Sigma) \right) \psi.
\]
By definition, we have the decomposition $\psi = \psi_D+ \psi_\Sigma$. In addition, we have
\begin{align*}
     \left( H_{\rm max}^\sharp - \Sigma \right) \psi_\Sigma & = 0,
\end{align*}
hence $\psi_\Sigma \in \widetilde{\cS_\Sigma}$. This gives a decomposition
\[
     \cD_{\rm max}^\sharp = \cD_D^\sharp + \widetilde{\cS_\Sigma},  \quad \psi = \psi_D + \psi_\Sigma.
\]
For $\psi = \psi_D + \psi_\Sigma$  a smooth function, we define the {\em regularized} trace-map
\[
    \boxed{ \widetilde{\Tr}(\psi) := \left( \gamma^D \psi,  \gamma^N \psi_D \right).}
\]
The term ``regularized'' comes from the fact that only the $\psi_D$ part appears in the Neumann trace. Since $\gamma^D \psi_D = 0$, the Dirichlet trace is also $\gamma^D \psi = \gamma^D \psi_\Sigma$.

\begin{lemma}
    The map $\widetilde{\Tr}$ can be extended as a bounded map from $\cD_{\rm max}^\sharp$ (equipped with the graph norm) to the boundary space
    \[
        \widetilde{\cH_b} := H^{-1/2}(\Gamma) \times H^{1/2}(\Gamma).
    \]
    This extension $\widetilde{\Tr} : \cD_{\rm max}^\sharp \to \widetilde{\cH_b}$ is surjective. The following Green's identity holds: for all $\phi, \psi \in \cD_{\rm max}^\sharp$, we have
    \[
        \bra \phi, H_{\rm max}^\sharp \psi \ket - \bra H_{\rm max}^\sharp \phi,  \psi \ket
         = \bra \gamma^D \phi, \gamma^N \psi_D \ket_{H^{-1/2}, H^{1/2}} - \bra \gamma^N \phi_D, \gamma^D \psi \ket_{H^{1/2}, H^{-1/2}}.
    \]
\end{lemma}
We refer to~\cite[Theorem 8.4.1]{Behrndt_2020} for the proof. Here, $\bra \cdot, \cdot \ket_{H^{-1/2}, H^{1/2}}$ denotes the duality product. 

We introduce the symplectic form $\widetilde{\omega} : \widetilde{\cH_b} \times \widetilde{\cH_b} \to \C$ defined by
\[
    \forall (f,f'), (g,g') \in \widetilde{\cH_b}, \quad \widetilde{\omega}((f,f'), (g,g')) := \bra f, g' \ket_{H^{-1/2}, H^{1/2}} - \bra f', g \ket_{H^{1/2}, H^{-1/2}}.
\]
One can check that $\left(\widetilde{\cH_b}, \widetilde{\omega}\right)$ is a symplectic Hilbert space. With this, the Green's identity takes the form
\[
    \forall \phi, \psi \in \cD_{\rm max}^\sharp, \quad \bra \phi, H_{\rm max}^\sharp \psi \ket - \bra H_{\rm max}^\sharp \phi,  \psi \ket
    = \widetilde{\omega}\left( \widetilde{\Tr}(\phi), \widetilde{\Tr}(\psi) \right).
\]
Unlike the previous $\Tr$ map in~\eqref{eq:def:Tr_S}, the $\widetilde{\Tr}$ map now depends on the operator $H_{\rm max}^\sharp$ and on the choice of $\Sigma$.

\subsubsection{General self-adjoint extensions}

Since the trace map $\widetilde{\Tr}$ is continuous and onto, one can repeat the arguments of Theorem~\ref{th:self_adjoint_extensions_are_Lagrangian_planes}. We obtain the following.
\begin{theorem} \label{th:self_adjoint_extensions_are_Lagrangian_planes_Schrodinger_2}
    Let $\cD^\sharp$ be a domain satisfying $\cD_{\rm min}^\sharp \subset \cD^\sharp \subset \cD_{\rm max}^\sharp$ and let $\ell := \widetilde{\Tr}( \cD^\sharp)$. Then the adjoint domain is $(\cD^\sharp)^* = \widetilde{\Tr}^{-1} \left(\ell^\circ \right)$. \\
    In particular, $(H^\sharp_{\rm max}, \cD^\sharp)$ is a self-adjoint extension iff
    \[
        \exists \ell \in \Lambda(\widetilde{\cH_b}) \quad \text{so that} \quad \cD^\sharp = \widetilde{\Tr}^{-1} (\ell).
    \]
\end{theorem}
This gives a one-to-one correspondence between all self-adjoint extensions of $(- \Delta + V)$ on the half-tube $L^2(\Omega^+)$, and the Lagrangian planes of $\left(\widetilde{\cH_b}, \widetilde{\omega}\right)$.  \\

The symplectic space $(\widetilde{\cH_b},\widetilde{\omega})$ satisfies Assumption B (hence A). Indeed, let $V : H^{1/2}(\Gamma) \to H^{-1/2}(\Gamma)$  be the map such that for all $f \in H^{1/2}(\Gamma)$ and all $g \in H^{-1/2}(\Gamma)$, we have
\begin{equation} \label{eq:def:V}
    \bra f, g \ket_{H^{1/2}, H^{-1/2}} = \bra f, V^* g \ket_{H^{1/2}} 
    =  \bra V f, g \ket_{H^{-1/2}}.
\end{equation}
The existence of such map $V$ comes from Riesz' Lemma, and we can check that $V$ is unitary. This time, we have $J = \begin{pmatrix}
    0 & V \\ - V^* & 0
\end{pmatrix}$, which satisfies Assumption B.

\medskip

In particular, the Lagrangian planes of $(\widetilde{\cH_b}, \widetilde{\omega})$ are in one-to-one correspondence with the unitaries $\widetilde{\cU}$ of $H^{-1/2}(\Gamma)$ with
\[
\widetilde{\ell} = \left\{ \begin{pmatrix} 1 \\ \ri V^*  \end{pmatrix}  f +  \begin{pmatrix} 1 \\ -\ri V^* \end{pmatrix} \widetilde{ \cU } f, \quad f \in H^{-1/2}(\Gamma) \right\}.
\]
As the Hilbert space $H^{-1/2}(\Gamma)$ is unitary equivalent to $L^2(\Gamma)$ and to $H^{1/2}(\Gamma)$, one has similar one-to-one correspondence replacing $H^{-1/2}(\Gamma)$ by $L^2(\Gamma)$ or $H^{1/2}(\Gamma)$.


\subsubsection{The planes $\widetilde{\ell(E)}$}

Let us now focus on the boundary traces of $\widetilde{\cS_E} = {\rm Ker}(H_{\rm max}^\sharp - E)$. For $E \in \R$, we introduce the planes
\[
    \widetilde{\ell}(E) := \widetilde{\Tr}( \widetilde{\cS_E}) \quad \subset \widetilde{\cH_b}.
\]

\begin{remark} \label{rem:creation_essential_spectrum}
    The plane $\ell_D := \{ 0 \} \times H^{1/2}(\Gamma)$ is Lagrangian, and corresponds to the Dirichlet extension. However, the plane $\ell_\Sigma := H^{-1/2}(\Gamma) \times \{ 0 \}$ is Lagrangian, but does not correspond to the Neumann extension. It rather corresponds to $\widetilde{\ell}(E = \Sigma)$. The self-adjoint extension corresponding to the Lagrangian plane $\ell_\Sigma$ has $\Sigma$ as an eigenvalue of infinite multiplicities (hence $\Sigma$ is in its essential spectrum).
\end{remark}

The counterpart of Theorem~\ref{th:Lagrangian_planes_Schrodinger_lE} is the following. 
\begin{theorem} \label{th:Lagrangian_planes_Schrodinger_tildelE}
    For all $E \in \R \setminus \sigma_{\rm ess} \left( H^\sharp_D  \right)$, $\widetilde{\ell}(E)$ is a Lagrangian plane of $\left( \widetilde{\cH_b}, \widetilde{\omega} \right)$.
\end{theorem}
Unlike Theorem~\ref{th:Lagrangian_planes_Schrodinger_lE}, only the essential spectrum of $H^\sharp_D$ matters. This result is independent of the value of $V$ on the left side $\Omega^-$ (see Remark~\ref{rem:ell+_and_ell-}).

\begin{proof}
    First, it is clear that $\widetilde{\ell}(E)$ is isotropic (see the proof of Theorem~\ref{th:Lagrangian_planes_Schrodinger_lE}).
    
   \underline{Case 1: $E$ is in the resolvent set.} Let us first prove the result for $E \in \R \setminus \sigma \left( H^\sharp_D  \right)$. In this case, the proof follows the lines of~\cite[Proposition 8.4.4]{Behrndt_2020}. 

  Let $(f, f')  \in \widetilde{\ell}(E)$, and let $\psi_E \in \widetilde{\cS_E}$ be such that $\widetilde{\Tr}(\psi_E) = (f,f')$. Write $\psi_E = \psi_D + \psi_\Sigma$ with $\psi_D \in \cD_D^\sharp$ and $\psi_\Sigma \in \widetilde{\cS_\Sigma}$. Applying $(H_{\rm max}^\sharp - E)$ shows that
    \[
        0 = (H_D^\sharp - E) \psi_D + (\Sigma - E) \psi_\Sigma, \quad \text{hence} \quad \psi_D = (E - \Sigma ) (H_D^\sharp - E)^{-1} \psi_\Sigma.
    \]
    So
    \[
        \psi_E = \left( 1 + (E - \Sigma ) (H_D^\sharp - E)^{-1}  \right) \psi_\Sigma.
    \]
    In particular, we have
    \begin{equation} \label{eq:f-f'}
        f = \gamma^D \psi_E = \gamma^D \psi_\Sigma, \quad \text{and} \quad f' = \gamma^N \psi_D =  (E - \Sigma )  \gamma^N (H_D^\sharp - E)^{-1} \psi_N.
    \end{equation}
    Recall that $\widetilde{\Tr}$ is bijective from $\cD_{\rm max}^\sharp$ to $H^{-1/2}(\Gamma) \times H^{1/2}(\Gamma)$, and set $G_\Sigma (f) := \widetilde{\Tr}^{-1} (f, 0)$. By decomposing $G_\Sigma(f)$ as $G_\Sigma(f) = g_D + g_\Sigma$, one must have $\gamma^N(g_D) = 0$ with $g_D \in \cD_D^\sharp$. This implies $g_D = 0$ by the unique continuation principle, so $G_\Sigma(f) = g_\Sigma \in \widetilde{\cS_\Sigma}$. In other words, $G_\Sigma$ is the map from $H^{-1/2}(\Gamma)$ to
  $\widetilde{\cS_\Sigma}$ so that $\gamma^D (G_\Sigma f) = f$ (this map is called the $\gamma$-field). 
  
  The first equation of~\eqref{eq:f-f'} reads $\psi_\Sigma = G_\Sigma f$, and the second shows that
    \[
       f' = M(E) f, \quad \text{with} \quad M(E) := (E - \Sigma) \gamma^N  (H_D^\sharp - E)^{-1} G_\Sigma.
    \]
    The map $M(E) : H^{-1/2}(\Gamma) \to H^{1/2}(\Gamma)$ is called the Weyl $M$-function. 
    
    Conversely, let $f \in H^{-1/2}(\Gamma)$, and set $f' = M(E) f$. By defining $\psi_E := \left( 1 + (E - \Sigma ) (H_D^\sharp - E)^{-1}  \right) G_\Sigma f$, we can check that $\psi_E \in \widetilde{\cS_E}$ and $\widetilde{\Tr}(\Psi_E) = (f, f')$. So $(f, f') \in \widetilde{\ell_E}$. This proves that
    \[
        \widetilde{\ell}(E) = \left\{ (f, M(E) f), \quad f \in H^{-1/2}(\Gamma)  \right\},
    \]
    that is $\widetilde{\ell^+}(E)$ is the graph of the map $M(E)$. 
        
    The map $M(E)$ is a bounded operator from $H^{-1/2}(\Gamma)$ to $H^{1/2}(\Gamma)$. The fact that $\widetilde{\ell}(E)$ is isotropic is equivalent to the fact that $M(E)$ is symmetric, in the sense
    \[
        \forall f, g \in H^{-1/2}(\Gamma), \quad \bra f, M(E) g \ket_{H^{-1/2}, H^{1/2}} = \bra M(E) f,  g \ket_{H^{1/2}, H^{-1/2}} .
    \]
    Now, let $(g,g') \in \left( \widetilde{\ell} (E)\right)^\circ$. For all $f \in H^{-1/2}(\Gamma)$, we have
    \[
        0 = \widetilde{\omega}( (f, M(E)f), (g,g') ) = \bra f, g' \ket_{H^{-1/2}, H^{1/2}} - \bra M(E) f,  g \ket_{H^{1/2}, H^{-1/2}}.
    \]
    Comparing with the previous line, this implies
    \[
        \forall f \in H^{-1/2}(\Gamma), \quad \bra f, (g' - M(E) g) \ket_{H^{-1/2}, H^{1/2}} = 0,
    \]
    hence $g' = M(E) g$, and $(g, g') \in \widetilde{\ell}(E)$. So $\widetilde{\ell}(E) = \left( \widetilde{\ell}(E) \right)^\circ$ as wanted.\\

    \underline{Case 2: $E$ is an eigenvalue of $H_D^\sharp$}. We now prove the result when $E \in \sigma(H_D^\sharp) \setminus \sigma_{\rm ess}(H_D^\sharp)$ is an isolated eigenvalue of $H_D^\sharp$ of finite multiplicity. This case is novel to the best of our knowledge.
    
    As before, we consider $\psi_E \in \widetilde{\cS_E}$ and set $\widetilde{\Tr}(\psi_E) = (f,f')$. We decompose $\psi_E$ as $\psi_E = \psi_D + \psi_\Sigma$, and we get again
    \[
        (H_D^\sharp - E) \psi_D = (E - \Sigma) \psi_\Sigma.
    \]
    This time, the operator $(H_D^\sharp - E)$ is non invertible. We consider the decomposition $L^2(\Omega^+) = K \oplus K^\perp$ with $K := {\rm Ker}(H_D^\sharp - E)$. We deduce first that $\psi_\Sigma \in K^\perp$, and then that
    \[
        \psi_D =  (E - \Sigma) (H_D^\sharp - E)^\dagger \psi_\Sigma + \psi_K,
    \]
    for some $\psi_K \in K$. Here, $(H_D^\sharp - E)^\dagger$ denotes the pseudo-inverse of $(H_D^\sharp - E)$. It is a bounded operator on $L^2(\Omega^+)$, as $E$ is an isolated eigenvalue. Taking boundary trace shows that
    \[
        f = \gamma^D \psi_\Sigma \quad \text{and} \quad f' = M(E) f + \gamma^N \psi_K, \quad \text{with} \quad M(E) := (E - \Sigma) \gamma^N  (H_D^\sharp - E)^{\dagger} G_\Sigma.
    \]
    We deduce that
    \[
        \widetilde{\ell_E} \subset \left\{  (f, M(E) f), \ f \in \gamma^D K^\perp \right\} + \{ (0, \gamma^N \psi_K), \ \psi_K \in K \}.
    \]
    Conversely, given $f \in  \gamma^D K^\perp$ and $\psi_K \in K$, the function
    \[
        \psi_\Sigma := \left( 1 + (E - \Sigma)(H_D^\sharp - E)^\dagger  \right) G_\Sigma f + \psi_K,
    \]
    is in $\widetilde{\cS_E}$ (we use here that $(H_D^\sharp - E) (H_D^\sharp - E)^\dagger = P_K^\perp$, where $P_K^\perp$ is the orthogonal projection on $K^\perp$, and that $G_\Sigma f$ is in $K^\perp$), and satisfies $\widetilde{\Tr}(\psi_\Sigma)= (f, M(E) f) + (0, \gamma^N \Psi_K)$. So we have the equality
    \[
        \widetilde{\ell_E} = \left\{  (f, M(E) f), \ f \in \gamma^D K^\perp \right\} + \{ (0, \gamma^N \psi_K), \ \psi_K \in K \}.
    \]
    From the isotropy of $\ell(E)$, we deduce that the operator $M(E)$ is symmetric on $\gamma^D K^\perp$, in the sense that 
    \[
    \forall f, g \in \gamma^D K^\perp, \quad \bra f, M(E) g \ket_{H^{-1/2}, H^{1/2}} = \bra M(E) f,  g \ket_{H^{1/2}, H^{-1/2}} .
    \]
    In addition, we also have
    \[
        \forall f \in \gamma^D K^\perp, \quad \forall \psi_K \in K, \quad   \bra f, \gamma^N \psi_K \ket_{H^{-1/2}, H^{1/2}} = 0.
    \]
    Since $ \gamma^D K^\perp$ is of codimension ${\rm dim}(K)$ in $H^{-1/2}(\Gamma)$ while $\{ \gamma^N \psi_K, \ \psi_K \in K\}$ is of dimension ${\rm dim}(K)$, we deduce that if $h \in H^{1/2}(\Gamma)$ satisfies $\bra f, h \ket_{H^{-1/2}, H^{1/2}} = 0$ for all $f \in \gamma^D K^\perp$, then $h = \gamma^N \psi_K$ for some $\psi_K \in K$.
    
    Let us finally prove that $\widetilde{\ell_E}$ is Lagrangian. Let $(g,g') \in \left( \widetilde{\ell} (E)\right)^\circ$, and let $\psi_g \in \cD_{\rm max}^\sharp$ be such that $\widetilde{\Tr}(\psi_g) = (g,g')$. We write $\psi_g = \psi_K + \psi_K^\perp$ with $\psi_K = P_K \psi_g$. Taking Dirichlet trace shows that $g = \gamma^D \psi_g = \gamma^D \psi_K^\perp \in \gamma^D K^\perp$. We set $h := g' - M(E) g$. We have, for all $f \in \gamma^D K^\perp$,
    \begin{align*}
    0 & = \widetilde{\omega}( (f, M(E)f), (g,g') ) = \bra f, g' \ket_{H^{-1/2}, H^{1/2}} - \bra M(E) f,  g \ket_{H^{1/2}, H^{-1/2}}  \\
     & = \bra f, g' - M(E) g \ket_{H^{-1/2}, H^{1/2}} = \bra f, h \ket_{H^{-1/2}, H^{1/2}}.
    \end{align*}
    We deduce that $h = \gamma^D \psi_K'$ for some $\psi_K' \in K$, hence $g' = M(E) g + \gamma^D \psi_K'$. So $(g,g') \in \widetilde{\ell_E}$, and $\left( \widetilde{\ell_E} \right)^\circ =  \widetilde{\ell_E} $ is Lagrangian.
\end{proof}

\subsubsection{Concluding remarks}

The use of the boundary trace $\widetilde{\cH_b}$ is suitable to study Schrödinger operators on the half-tubes $\Omega^+$ or $\Omega^-$. Indeed, one can detect that $E$ is an eigenvalue for a general self-extension $(H^\sharp, \cD^\sharp)$ as the crossing of the Lagrangian planes $\widetilde{\ell}(E)$ and $\ell^\sharp$ in $\widetilde{\cH_b}$. In addition, given a family of self-adjoint operators parametrized by $\ell_t^\sharp$, one can compute the spectral flow of this family as the Maslov index $\Mas(\widetilde{\ell}(E), \ell^\sharp_t, \TT^1)$. Since $(\widetilde{\cH_b}, \widetilde{\omega})$ satisfies Assumption B (hence A), this can be done using unitaries.

\medskip

This setting is however not suitable to study the junction case, or more generally, to study operators on the whole tube $\Omega$. The reason is the following. Let us consider the corresponding objects on the left tube $\Omega^-$. The trace operators $\widetilde{\Tr}^-$ and $\widetilde{\Tr}^+$ depend on the left and right operators $H_D^{\sharp, -}$ and $H_D^{\sharp, +}$ respectively. In particular, they are unrelated! There is no analogue to Lemma~\ref{lem:E_eigenvalue_of_H} in this setting: the crossing of $\widetilde{\ell^+}(E)$ and $\widetilde{\ell^-}(E)$ does not imply that $E$ is an eigenvalue of the bulk operator $H$. For instance, for $E = \Sigma$, we have $\widetilde{\ell^+}(\Sigma) = \widetilde{\ell^-}(\Sigma) = H^{-1/2}(\Gamma) \times \{ 0 \}$, but $\Sigma$ can be in the resolvent set of $H$. This is the reason why we chose to work in the $H^2(\Omega^\pm)$ setting, and to use the trace operator $\Tr$.


 \bibliographystyle{my-alpha}
\bibliography{biblio}

\end{document}